\documentclass{article}
\usepackage{a4wide}
\usepackage{amsmath}
\usepackage{blindtext}
\usepackage{amssymb}
\usepackage{graphicx}
\usepackage{enumerate}
\usepackage{amsthm}
\usepackage{mathtools}
\usepackage{geometry}
\usepackage{color}
\usepackage{tikz-cd}
\usepackage{tikz}
\usepackage{ gensymb }
\usepackage{shuffle}
\usepackage{float}
\usepackage{calc}
\usepackage{xcolor}
\usepackage{hyperref}
\usepackage{caption}
\usepackage{epigraph}
\usepackage{soul}

\usepackage{datetime}
\renewcommand{\today}{\ifcase \month \or January\or February\or March\or April\or May \or June\or July\or August\or September\or October\or November\or December\fi, \number \year} 
  
\usepackage[utf8x]{inputenc}
\usepackage[T2A,T1]{fontenc}
\usepackage[russian,english]{babel}
\usepackage{palatino}

\newcommand{\N}{\mathbb{N}}
\newcommand{\Z}{\mathbb{Z}} 
\newcommand{\Q}{\mathbb{Q}} 
\newcommand{\R}{\mathbb{R}}

\newcommand{\Ha}{\mathfrak{H}}

\newcommand{\astq}{\ast_q}

\newcommand{\negphantom}[1]{\settowidth{\dimen0}{#1}\hspace*{-\dimen0}}

\newcommand\restr[2]{{\left.\kern-\nulldelimiterspace #1 \right|_{#2}}}

\newcommand{\bbracket}[1]{\begin{bmatrix}#1\end{bmatrix}}

\newcommand{\fgen}[3]{\begin{vmatrix}#1\end{vmatrix}_{#2,#3}}
\newcommand{\sgen}[1]{\fgen{#1}{S}{R}}
\newcommand{\pgen}[1]{\begin{Vmatrix}#1\end{Vmatrix}_{S,R}}
\newcommand\vertarrowbox[3][3ex]{%
  \begin{array}[t]{@{}c@{}} #2 \\
  \left\uparrow\vcenter{\hrule height #1}\right.\kern-\nulldelimiterspace\\
  \makebox[0pt]{\scriptsize#3}
  \end{array}%
}

\newcommand{\diam}{\,\overline{\diamond}\,}
\newcommand{\shuf}{\,\overline{\shuffle}\,}
\newcommand{\astt}{\,\overline{\ast}\,}

\newcommand{\stuffle}{\ast}
\newcommand{\Li}{\operatorname{Li}}
\newcommand{\rank}{\operatorname{rank}}
\newcommand{\MD}{\mathcal{MD}}
\newcommand{\qMZ}{\operatorname{q}\!\mathcal{MZ}}
\newcommand{\BD}{\mathcal{BD}}

\newcommand{\s}{\textbf{s}}
\newcommand{\del}{\partial}

\newenvironment{changemargin}[2]{%
\begin{list}{}{%
\setlength{\topsep}{0pt}%
\setlength{\leftmargin}{#1}%
\setlength{\rightmargin}{#2}%
\setlength{\listparindent}{\parindent}%
\setlength{\itemindent}{\parindent}%
\setlength{\parsep}{\parskip}%
}%
\item[]}{\end{list}}

\newtheorem{theorem}{Theorem}[section]
\newtheorem*{conjecture*}{Conjecture}
\newtheorem{lemma}[theorem]{Lemma}
\newtheorem{proposition}[theorem]{Proposition}
\newtheorem{suspicion}{Suspicion}
\newtheorem{corollary}[theorem]{Corollary}

\title{Zeta}
\author{Abel Vleeshouwers (s4448588)}
\date{Fall 2019}

\begin{document}














\renewcommand{\abstractname}{}
\begin{titlepage}

\newcommand{\HRule}{\rule{\linewidth}{0.5mm}} 

\center 
\begin{figure}[ht]
\begin{center}
\includegraphics[height=6cm]{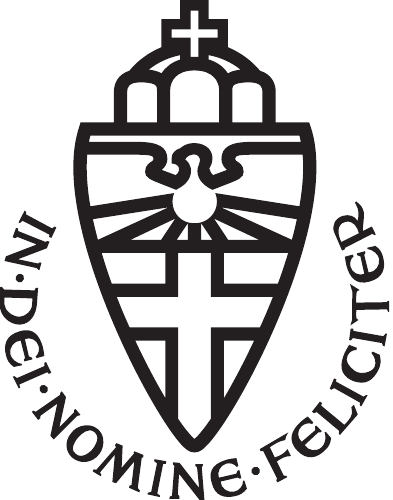}
\end{center}
\end{figure}
\smallskip
\textsc{\LARGE Radboud University}\\[0.3cm] 
\textsc{\large Faculty of Science}\\[1cm]

\HRule \\[0.4cm]
{ \Huge   Multiple zeta values and their $q$-analogues}\\[0.4cm] 
\HRule \\[1.5cm]
\textsc{\Large Master thesis in Mathematics}\\[0.5cm] 

\large
{\LARGE{Abel Vleeshouwers}}\\

\bigskip
\bigskip

\emph{Supervisor:\hphantom{blblasfgsfdasdfdfadsfdsafsadsafsflbalajlkabj}Second reader:}\\
\Large{prof. dr. Wadim Zudilin \hphantom{blblalbalajlkabj} dr. Maarten van Pruijssen}\\

\vfill
\bigskip
\bigskip

\bigskip
\bigskip

\bigskip

{\monthname \;\number\year}\\[3cm] 

\end{titlepage}
\thispagestyle{empty}
\clearpage

\hphantom{.}
\thispagestyle{empty}
\clearpage

\tableofcontents
\thispagestyle{empty}
\clearpage


\epigraph{\textit{Dikke huisjesslak,\\ ook jij beklimt de Fuji\\-- maar langzaam, langzaam.}}{--- Issa}
\section*{Acknowledgements}
\addcontentsline{toc}{section}{Acknowledgements}
Ik wil graag hen bedanken die me het jaar door gesteund hebben en er voor me waren. Zij hebben er voor gezorgd dat, zelfs in de tijden van \textit{afstand houden}, ik gelukkig en met tevredenheid deze scriptie heb mogen schrijven.\\

\noindent Verder wil ik Isabelle en Maarten bedanken voor het proeflezen van de scriptie en voor hun feedback.\\

\noindent Ich m\"ochte auch Henrik danken f\"ur seine Suggestionen und Hilfe.

\vspace{1cm}
\noindent  \foreignlanguage{russian}{И, наконец, я бы хотел
поблагодарить Вадима за всестороннюю и неоценимую поддержку.}
\begin{center}\section*{Abstract}\end{center}
\addcontentsline{toc}{section}{Abstract}

\begin{changemargin}{1.1cm}{1.1cm}

\noindent We explore the theory of multiple zeta values (MZVs) and some of their $q$-$\negphantom{ }$ generalisations. Multiple zeta values are numerical quantities that satisfy several combinatorial relations over the rationals. These relations include two multiplicative relations, which arise naturally from comparison of the MZVs with an underlying algebraic structure. We generalise these concepts by introducing the parameter $q$ in such a way that as $q\to 1^-$ we return to the ordinary MZVs. Our special interest lies in two $q$-models recently introduced by H. Bachmann. He further conjectures that the $\Q$-spaces generated by these $q$-generalisations coincide. In this thesis we establish a particular case of Bachmann's conjecture.\\

\end{changemargin}

\section{Multiple zeta values}

The Riemann zeta function has always been concept of great study. It is defined as the series
\begin{equation*}
    \zeta(s)=\sum_{n=1}^\infty\frac{1}{n^s}
\end{equation*}
for complex $s$. The series is convergent whenever $\operatorname{Re}(s)>1$. The zeta function plays an important role in number theory, due to its connection to prime numbers
\begin{equation*}
    \zeta(s)=\prod_{p \text{ prime}}\frac{1}{1-p^{-s}}.
\end{equation*}
When we consider the integer values $\zeta(m)$ of the Riemann zeta function for $m$ positive. Some of these values are well-known constants, such as
\begin{align*}
    \zeta(2)=\frac{1}{1^2}+\frac{1}{2^2}+\frac{1}{3^2}+\frac{1}{4^2}+\dots=\frac{\pi^2}{6}
\end{align*}
and Ap\'ery's constant
\begin{equation*}
    \zeta(3)=\frac{1}{1^3}+\frac{1}{2^3}+\frac{1}{3^3}+\frac{1}{4^3}+\dots\approx 1.202056903\dots\;.
\end{equation*}
Generally, more is known for even zeta values $\zeta(2m)$. For example, we have the closed representation
\begin{equation*}
    \zeta(2m)=\frac{(-1)^{m+1}B_{2m}(2\pi)^{2m}}{2(2m)!},
\end{equation*}
where the Bernoulli numbers $B_{k}$ are defined as the coefficients of the generating function of $\frac{x}{e^x-1}=\sum_{k=0}^\infty \frac{B_k}{k!}x^k$.\\

In the $1990$s, Hoffmann and Zagier introduced a multi-variable variant of the Riemann zeta function. This series is called the \textit{multiple zeta function} and is defined for a multi-index $\s=(s_1,\dots,s_l)$ of complex numbers with $\operatorname{Re}(s_1)>1$ and $\operatorname{Re}(s_2),\dots,\operatorname{Re}(s_l)>0$. It is defined as the series
\begin{align*}
    \zeta(\s)=\zeta(s_1,\dots,s_l)&\coloneqq\sum_{n_1>n_2>\dots>n_l>0}\frac{1}{n_1^{s_1}n_2^{s_2}\cdots n_l^{s_l}}.
\end{align*}
The values at integral points $(s_1,\dots,s_l)$ are called \textit{multiple zeta values}, or \textit{MZVs}. In this thesis we will only be interested in these multiple zeta values. We call a multi-index $\s=(s_1,\dots,s_l)$ \textit{admissible} if $s_1>1$ and $s_j>0$ for $2\leq j\leq l$. The \textit{length} or \textit{depth} of a multiple zeta value $\zeta(s_1,\dots,s_l)$ is the integer $l$, and its \textit{weight} is the value $s_1+\dots+s_l$.

\begin{lemma}
For integers $s_1>1$ and $s_2,\dots,s_l>0$, the series
\begin{equation*}
    \sum_{n_1>\dots>n_l>0}\frac{1}{n_1^{s_1}\cdots n_l^{s_l}},
\end{equation*}
defining the multiple zeta values, converges (absolutely).
\end{lemma}

\noindent\textit{Proof.} If $l=1$, the series is the ordinary Riemann zeta function $\zeta(s)=\sum_{n=1}^\infty 1/n^s$, which converges whenever $s>1$. Now suppose the statement is true for some integer $l$. Then, by this assumption, the quantity $\sum_{n_1>\dots>n_l>M}\frac{1}{n_1^{s_1}\cdots n_l^{s_l}}$ converges. After repeated use of the inequality
\begin{equation*}
    \sum_{n>M}\frac{1}{n^s}\leq \frac{1}{(s-1)M^{s-1}}
\end{equation*}
for $s>1$, and the notion that $s_1+\dots+s_i-i>1$, we obtain
\begin{align*}
    \sum_{n_1>n_2>\dots>n_l>M}\frac{1}{n_1^{s_1}\cdots n_l^{s_l}}&=\sum_{n_l>M}^\infty\frac{1}{n_l^{s_l}}\left(\sum_{n_{l-1}>n_l}^\infty\frac{1}{n_{l-1}^{s_{l-1}}}\left(\cdots\sum_{n_2>n_3}^\infty\frac{1}{n_2^{s_2}}\left(\sum_{n_1>n_2}^\infty\frac{1}{n_1^{s_1}}\right)\right)\cdots\right)\\
    &\leq\sum_{n_l>M}^\infty\frac{1}{n_l^{s_l}}\left(\sum_{n_{l-1}>n_l}^\infty\frac{1}{n_{l-1}^{s_{l-1}}}\left(\cdots\sum_{n_2>n_3}^\infty\frac{1}{n_2^{s_2}}\cdot\frac{1}{(s_1-1)n_2^{s_1-1}}\right)\cdots\right)\\
    &=\frac{1}{s_1-1}\sum_{n_l>M}^\infty\frac{1}{n_l^{s_l}}\left(\sum_{n_{l-1}>n_l}^\infty\frac{1}{n_{l-1}^{s_{l-1}}}\left(\cdots\sum_{n_2>n_3}^\infty\frac{1}{n_2^{s_1+s_2-1}}\right)\cdots\right)\\
    &\hphantom{hello}\vdots\\
    &\leq\frac{1}{(s_1-1)(s_1+s_2-2)\cdots(s_1+\dots+s_l-l)M^{s_1+\dots+s_l-l+1}}.
\end{align*}
But then, because $s_1+\dots+s_l-l>1$, the series
\begin{equation*}
    \zeta(s_1,\dots,s_l,s_{l+1})\leq \frac{1}{(s_1-1)(s_1+s_2-2)\cdots(s_1+\dots+s_l-l)}\sum_{n_{l+1}=1}^\infty \frac{1}{n_{l+1}^{s_1+\dots+s_l+s_{l+1}-l}}
\end{equation*}
converges.\qed\\

\noindent\textbf{Example.} The multiple zeta values of weight $3$ and $4$ are given by
\begin{align*}
    \zeta(2,1)&=1.202056903159594\ldots=\zeta(3)\\
\intertext{and}
    \zeta(2,1,1)&=1.082323233711138\ldots=\zeta(4),\\
    \zeta(3,1)&=0.270580808427785\ldots,\\
    \zeta(2,2)&=0.811742425283353\ldots.
\end{align*}

This example already shows that there exist some apparent relations between the multiple zeta values (note that we also seem to have $\zeta(4)=4\zeta(3,1)=\frac{4}{3}\zeta(2,2)$). The non-obvious identity
\begin{equation*}
    \zeta(3)=\zeta(2,1)
\end{equation*}
was already known by Euler. It turns out that there is an abundance of these relations. We will devote the upcoming section to give some examples of more generic identities between multiple zeta values.

\subsection{Examples}
\label{subsection:examplesmzvtheorems}
\begin{theorem}[Euler]
\label{eulertheorem}
For an integer $s\geq 3$ we have
\begin{equation*}
    \zeta(s)=\sum_{j=2}^{s-1}\zeta(j,s-j).
\end{equation*}
\end{theorem}
\begin{proof}
The statement is a particular instance of Theorem \ref{sumtheorem}.
\end{proof}
The next theorem is a weighted version of Theorem \ref{eulertheorem}. A proof can be found in \cite[Thm. 2.3, p. 19]{notes}.
\begin{theorem}
For an integer $s\geq 3$ we have
\begin{equation*}
    (s+1)\zeta(s)=\sum_{j=2}^{s-1}2^j\zeta(j,s-j).
\end{equation*}
\end{theorem}
Let $\s=(s_1,\dots,s_l)$ an admissible multi-index. We can rewrite $\s$ as
\begin{equation*}
    \s=(u_1+1,\underbrace{1,\dots,1}_{t_1-1\text{ times}},u_2+1,\underbrace{1,\dots,1}_{t_2-1\text{ times}},\dots,u_k+1,\underbrace{1,\dots,1}_{t_k-1\text{ times}}),
\end{equation*}
with integers $u_1,\dots,u_k,t_1,\dots,t_k>0$. We define the \textit{dual} of $\s$ to be the multi-index
\begin{equation*}
    \s^\dagger=(t_k+1,\underbrace{1,\dots,1}_{u_k-1\text{ times}},t_{k-1}+1,\underbrace{1,\dots,1}_{u_{k-1}-1\text{ times}},\dots,t_1+1,\underbrace{1,\dots,1}_{u_1-1\text{ times}}).
\end{equation*}
\noindent\textbf{Example.} The duals of the indices $\mathbf{r}=(4,1)$, $\s=(3,2,2)$ and $\mathbf{t}=(2,1,1)$ are given by
\begin{equation*}
    \mathbf{r}^\dagger=(3,1,1),\qquad\qquad\s^\dagger=(2,2,2,1),\qquad\qquad\mathbf{t}^\dagger=(4).
\end{equation*}
Furthermore, for every multi-index $\s$ we have
\begin{equation*}
    (\s^\dagger)^\dagger=\s.
\end{equation*}
\begin{theorem}[Ohno's relations]
\label{ohnosrelations}
Consider an admissible index $\s=(s_1,s_2,\dots,s_l)$, with its dual written as $\s^\dagger=(s'_1,s'_2,\dots, s'_m)$. Then for fixed integer $n\geq0$ we have
\begin{equation*}
    \sum_{\substack{e_1,\dots,e_l\geq 0\\e_1+\dots+e_l=n}}\zeta(s_1+e_1,s_2+e_2,\dots, s_l+e_l)=\sum_{\substack{e_1,\dots,e_m\geq 0\\e_1+\dots+e_m=n}}\zeta(s'_1+e_1,s'_2+e_2,\dots ,s'_m+e_m).
\end{equation*}
\end{theorem}
We will prove this result at a later stage, in Section \ref{subsection:qmultiplezetavalues}. When we apply Ohno's relations to the smallest case $n=0$, the identity we obtain between MZVs is known as \textit{duality}.
\begin{theorem}[Duality]
\label{duality}
For an admissible index $\s=(s_1,\dots,s_l)$, the duality
\begin{equation*}
    \zeta(\s)=\zeta(\s^\dagger)
\end{equation*}
holds.
\end{theorem}

\begin{theorem}[Sum theorem]
\label{sumtheorem}
For fixed integers $S>1$ and $l>1$ we have
\begin{equation*}
    \sum_{\substack{s_1>1,\;s_2,\dots,s_l>0\\s_1+\dots+s_l=S}}\zeta(s_1,\dots,s_l)=\zeta(S).
\end{equation*}
\end{theorem}
\begin{proof}
For fixed $S,l>1$ we consider the single-entry index $\s=(l+1)$. Its dual is then given by 
\begin{equation*}
    \s^\dagger=(2,\underbrace{1,1,\cdots,1}_{l-1 \text{ times}}).
\end{equation*}
Applying Ohno's relations with $n=S-l-1$ to $\s$, we will then obtain
\begin{align*}
    \zeta(S)&=\sum_{\substack{e_1,\dots,e_l\geq 0\\e_1+\dots+e_l=n}}\zeta(2+e_1,1+e_2\cdots ,1+e_l)\\
    &=\sum_{\substack{s_1>1,s_1,\dots,s_l>0\\s_1+\dots+s_l=n+l+1=S}}\zeta(s_1,\dots,s_l).\qedhere
\end{align*}
\end{proof}
In particular the Sum theorem states that the sum of all MZVs of fixed weight and length is constant, and only depends on its weight. If $l=2$, the Sum theorem is precisely Euler's theorem.
\begin{theorem}[Hoffman's relations]
\label{hoffmansrelations}
For an admissible index $\s=(s_1,\dots,s_l)$ we have
\begin{align*}
    \sum_{k=1}^l\zeta(s_1,\dots,s_{k-1},s_k+1,s_{k+1},\dots,s_l)=\sum_{\substack{k=1\\s_k\geq 2}}^l\sum_{j=0}^{s_k-2}\zeta(s_1,\dots,s_{k-1},s_k-j,j+1,s_{k+1},\dots,s_l).
\end{align*}
\end{theorem}
\begin{proof}
Let us consider Ohno's relations with $n=1$. Then the left-hand side of Ohno's relations becomes the left-hand side of our statement. Let us write 
\begin{equation*}
    \s=(s_1,\dots,s_l)=(u_1+1,\underbrace{1,\dots,1}_{t_1-1\text{ times}},\dots,u_p+1,\underbrace{1,\dots,1}_{t_p-1\text{ times}}),
\end{equation*}
with its dual
\begin{equation*}
    \s^\dagger=(s'_1,\dots,s'_m)=(t_p+1,\underbrace{1,\dots,1}_{u_p-1\text{ times}},\dots,t_1+1,\underbrace{1,\dots,1}_{u_1-1\text{ times}}).
\end{equation*}
We will apply the duality $\zeta(\s)=\zeta(\s^\dagger)$, stated in Theorem \ref{duality}, to each of the terms in the sum
\begin{equation*}
    \sum_{k=1}^l\zeta(s'_1,\dots,s'_{k-1},s'_k+1,s'_{k+1},\dots,s'_m).
\end{equation*}
Suppose that $s'_k$ is in one of the $p$ blocks $\dots,t_i+1,1,\dots,1,\dots$\;. If $s'_k=t_i+1$, then we have
\begin{align*}
    (t_p+1,&\underbrace{1,\dots,1}_{u_p-1\text{ times}},\dots,t_i+2,\underbrace{1,\dots,1}_{u_i-1\text{ times}},\dots,t_1+1,\underbrace{1,\dots,1}_{u_1-1\text{ times}})^\dagger\\
    &=(u_1+1,\underbrace{1,\dots,1}_{t_1-1\text{ times}},\dots,u_i+1,\underbrace{1,\dots,1}_{t_i\text{ times}},\dots,u_p+1,\underbrace{1,\dots,1}_{t_p-1\text{ times}}).
\end{align*}
On the other hand, if $s'_k=1$, with position
\begin{equation*}
    \dots,t_i+1,\underbrace{1,\dots,1}_{a_i-1\text{ times}},\vertarrowbox{1}{$s'_k$},\underbrace{1,\dots,1}_{b_i-1\text{ times}},\dots\;,
\end{equation*}
we get
\begin{align*}
    (t_p+1,&\underbrace{1,\dots,1}_{u_p-1\text{ times}},\dots,t_i+1,\underbrace{1,\dots,1}_{a_i-1\text{ times}},2,\underbrace{1,\dots,1}_{b_i-1\text{ times}},\dots,t_1+1,\underbrace{1,\dots,1}_{u_1-1\text{ times}})^\dagger\\
    &=(u_1+1,\underbrace{1,\dots,1}_{t_1-1\text{ times}},\dots,b_i,a_i,\underbrace{1,\dots,1}_{t_i-1\text{ times}},\dots,u_p+1,\underbrace{1,\dots,1}_{t_p-1\text{ times}}).
\end{align*}
When we consider all positions for $s'_k$ in this block, we see that the cases where $s'_k=1$ correspond to the cases $\dots s_{k-1},s_k-j,j+1,s_{k+1},\dots$ for $1\leq j\leq s_k-2$. The case where $s'_k=t_i+1$ corresponds to the case $j=0$.
\end{proof}
The MZVs also satisfy multiplicative relations, which can be connected to an algebraic structure. We will discuss these relations and this connection in the next chapter.
\section{The algebra of multiple zeta values}
\label{chapter:MZV}
Let us directly calculate the product between two multiple zeta values
\begin{align}
\label{eq:productmzvs}
    \zeta(s)\zeta(t)&=\sum_{n,m>0}\frac{1}{n^sm^t}.
\end{align}
Instead of summing over all pairs $n,m$, we can consider the cases where $n>m$, $m>n$ and $n=m$. The first two cases correspond to the zeta values $\zeta(s,t)$ and $\zeta(t,s)$, respectively, and the last case corresponds to the value $\zeta(s+t)$. This proves the following proposition.
\begin{proposition}[Stuffle product]
\label{stufflel=1}
For integers $s,t>1$ we have
\begin{equation*}
    \zeta(s)\zeta(t)=\zeta(s,t)+\zeta(t,s)+\zeta(s+t).
\end{equation*}
\end{proposition}
The product $\zeta(s)\zeta(t)$ can also be represented in an alternative way. This other representation in terms of MZVs follows from applying the identity
\begin{equation*}
    \frac{1}{n^sm^t}=\sum_{j=1}^{s+t-1}\left(\frac{{j-1\choose s-1}}{(n+m)^j m^{s+t-j}}+\frac{{j-1\choose t-1}}{(n+m)^j n^{s+t-j}}\right)
\end{equation*}
to the terms in (\ref{eq:productmzvs}). If we then set $k=n+m$ and apply the definition of the multiple zeta values to the pairs $k>n>0$ and $k>m>0$, we obtain the following relation.
\begin{proposition}[Shuffle product]
\label{shufflel=1}
For integers $s,t>1$ we have
\begin{equation*}
    \zeta(s)\zeta(t)=\sum_{j=1}^{s+t-1}\left({j-1\choose s-1}+{j-1\choose t-1}\right)\zeta(j,s+t-j)
\end{equation*}
\end{proposition}
Of course, this trick of splitting the product can be generalised to MZVs of arbitrary length. In this chapter we will connect these multiplicative relations to an algebraic structure in which these relations appear naturally.

\subsection{The algebra of multiple zeta values}
Consider the alphabet $X=\{x_0,x_1\}$, and all words generated by this alphabet. For a multi-index $\s=(s_1,\dots,s_l)$ we relate to it the word
\begin{equation*}
    x_\s=x_0^{s_1-1}x_1x_0^{s_2-1}x_1\cdots x_0^{s_l-1}x_1.
\end{equation*}
The degree $\operatorname{deg}(x_\s)=s_1+\dots+s_l$ is also called the weight of the word $x_\s$, and is denoted by $|\s|$. Similarly, the length $l(x_\s)=l(\s)=l$ is the number of times that the letter $x_1$ appears in $x_\s$. We say that $x_\s$ is \textit{admissible} if $s_1>1$.\\

Let $\Ha=\Q\langle X\rangle$ be the $\Q$-algebra generated by $X$, graded by weight. We define the spaces $\Ha^1=\Q\textbf{1}\oplus \Ha x_1$ and $\Ha^0=\Q\textbf{1}\oplus x_0\Ha x_1$, where $\textbf{1}$ denotes the empty word in $\Ha$. In other words, the space $\Ha^1$ consists of all $\Q$-linear combinations of the empty word and all words of the form $x_\s$, for some index $\s$. Every word $v\in \Ha^1$ can be written as $v=y_{s_1}y_{s_2}\cdots y_{s_l}$, where the letters $\{y_s\}_{s=1}^\infty$ are given by $y_s=x_0^{s-1}x_1$. The space $\Ha^0\subseteq\Ha^1$ is the subspace generated by all admissible words.\\

We define two multiplications. The first is called the \textit{shuffle product}, which is denoted by $\shuffle$, and is defined on $\Ha$. The second is the \textit{stuffle product}, denoted by $\ast$, and is defined only on $\Ha^1$. These multiplications are given by the following rules:\\
For any word $w$ we set
\begin{equation*}
    \mathbf{1}\shuffle w = w\shuffle \mathbf{1}=w,\qquad\qquad \mathbf{1}\ast w = w\ast \mathbf{1}=w.
\end{equation*}
Additionally, for any non-empty words $x_iv$ and $x_jw$, we recursively define
\begin{equation*}
    x_iv\shuffle x_jw=x_i(v\shuffle x_jw) + x_j(x_iv\shuffle w),
\end{equation*}
and for words $y_iw$ and $y_jv$, we define
\begin{equation*}
    y_iv\ast y_jw=y_i(v\ast y_jw) + y_j(y_iv\ast w) +y_{i+j}(v\ast w).
\end{equation*}
Finally, we extend this definition linearly over $\Ha$ and $\Ha^1$, respectively.\\

\noindent\textbf{Example.} We have
\begin{align*}
    x_0x_1\shuffle x_0x_1^2&=x_0(x_1\shuffle x_0x_1^2)+x_0(x_0x_1\shuffle x_1^2)\\
    &=\dots\\
    &=3x_0x_1x_0x_1^2+6x_0^2x_1^2+x_0x_1^2x_0x_1
\end{align*}
and
\begin{align*}
    x_0x_1\ast x_0x_1^2=y_2\ast y_2y_1&=y_2^2y_1+y_2(y_2\ast y_1) +y_4y_1\\
    &=2y_2^2y_1+y_2y_1y_2+y_2y_3+y_4y_1.
\end{align*}

Note that the subspace $\Ha^0$ is closed under $\shuffle$ and $\ast$. This gives rise to the algebras $\Ha^0_\shuffle=(\Ha^0,\shuffle)$ and $\Ha^0_\ast=(\Ha^0,\ast)$.\\

\begin{lemma}
The products $\shuffle$ and $\stuffle$ are commutative and associative.
\end{lemma}
\begin{proof}
We will prove the commutativity of $\shuffle$. Let $a,b$ be two words. We will show commutativity by induction on the sum of the lengths $l(a)+l(b)$. If either $a$ or $b$ is the empty word, then the statement follows by definition. So let us assume that we have $a=x_iv$ and $b=x_jw$ for some letters $x_i,x_j$ and words $v,w$. Then we can simply apply the inductive relation defining the product $\shuffle$ so that we obtain
\begin{align*}
    a\shuffle b&=x_iv\shuffle x_jw\\
    &=x_i(v\shuffle x_jw)+x_j(x_iv\shuffle w)\\
    &=x_i(x_jw\shuffle v)+x_j(w\shuffle x_iv)\\
    &=x_jw\shuffle x_iv\\
    &=b\shuffle a.
\end{align*}
This shows the commutativity of $\shuffle$. The other statements can proven in a similar way.
\end{proof}
For an admissible word $x_\s$ we define the map $\zeta\colon\Ha^0\to\R$ on the words
\begin{equation*}
    \zeta(\mathbf{1})=1,\qquad\qquad \zeta(x_\s)=\zeta(\s)=\zeta(s_1,s_2\dots,s_l),
\end{equation*}
and then extend this definition linearly over $\Ha^0$. With this definition we are able to describe the link between multiple zeta values and the algebras $\Ha^0_\shuffle$ and $\Ha^0_\ast$. The following theorems greatly motivate these shuffle and stuffle algebras.
\begin{theorem}
\label{homommzvstuffle}
The map $\zeta\colon \Ha^0_\ast\to \R$ is a homomorphism of algebras, i.e. for any admissible words $v,w \in \Ha^0$ we have 
\begin{equation*}
    \zeta(v\ast w)=\zeta(v)\cdot\zeta(w).
\end{equation*}
\end{theorem}
\begin{theorem}
\label{homommzvshuffle}
The map $\zeta\colon \Ha^0_\shuffle\to \R$ is a homomorphism of algebras, i.e. for any admissible words $v,w \in \Ha^0$ we have 
\begin{equation*}
    \zeta(v\shuffle w)=\zeta(v)\cdot\zeta(w).
\end{equation*}
\end{theorem}
We will prove Theorems \ref{homommzvstuffle} and \ref{homommzvshuffle} in Sections \ref{subsection:multiple harmonic sums} and \ref{subsection:generalised polylogarithms}, respectively. Additionally, the statements in these theorems can be combined and specified to give another relation between MZVs.
\begin{theorem}
\label{shufflestufflecombi}
For any admissible word $w\in\Ha^0$ we have
\begin{equation*}
    \zeta(x_1\shuffle w- x_1\ast w)=0
\end{equation*}
\end{theorem}
We will omit the proof, but it can be found in \cite[Thm. 3.3, pp. 28, 40]{notes}.\\

\noindent Let $\tau\colon \Ha\to\Ha$ be the anti-automorphism given by
\begin{equation*}
    \tau\colon x_{i_1}x_{i_2}\cdots x_{i_l}\mapsto x_{1-i_l}\cdots x_{1-i_2}x_{1-i_1}.
\end{equation*}
Note that this map is no different from the duality defined in Section \ref{subsection:examplesmzvtheorems}: $\tau( x_\s)=x_{\s^\dagger}$.
For example, we have $\tau(x_0^2x_1x_0x_1+x_0x_1)=x_0x_1x_0x_1^2+x_0x_1$. The final theorem we give, is the algebraic version of Theorem \ref{duality}.
\begin{theorem}
\label{dualitymzv}
For any admissible word $w\in\Ha^0$ we have 
\begin{equation*}
    \zeta(\tau w)=\zeta(w).
\end{equation*}
\end{theorem}

\noindent\textbf{Example.} The space generated by all MZVs of weight $4$ is equal to $\Q\langle \zeta(4)\rangle=\Q\langle \frac{\pi^4}{90}\rangle$. Namely, Theorem \ref{homommzvshuffle} and \ref{homommzvstuffle} give us the relations
\begin{align*}
    \zeta(2)\zeta(2)&=\zeta(x_0x_1)\zeta(x_0x_1)=\zeta(x_0x_1\shuffle x_0x_1)=\zeta(2x_0x_1x_0x_1+4x_0^2x_1^2)=2\zeta(2,2)+4\zeta(3,1)\\
    &=\zeta(y_2)\zeta(y_2)=\zeta(y_2\ast y_2)=\zeta(2y_2y_2+y_4)=2\zeta(2,2)+\zeta(4).
\end{align*}
This shows that $\zeta(3,1)=\frac{1}{4}\zeta(4)$. Additionally, Theorem $\ref{shufflestufflecombi}$ states that
\begin{equation*}
    0=\zeta(x_1\shuffle x_0^2x_1-y_1\ast y_3)=\zeta(x_0x_1x_0x_1+x_0^2x_1^2-x_0^3x_1)=\zeta(2,2)+\zeta(3,1)-\zeta(4),
\end{equation*}
so that $\zeta(2,2)=\zeta(4)-\zeta(3,1)=\frac{3}{4}\zeta(4)$. Finally, Theorem \ref{dualitymzv} gives us
\begin{equation*}
    \zeta(2,1,1)=\zeta(x_0x_1^3)=\zeta(x_0^3x_1)=\zeta(4).
\end{equation*}
Note that we could have also used Theorem \ref{shufflestufflecombi} again with $w=x_0x_1x_1$ to also obtain $\zeta(2,1,1)=\zeta(2,2)+\zeta(3,1)=\zeta(4)$.\\

This example shows that it is possible to deduce all possible $\Q$-linear relations in weight $4$ by only considering the relations obtained from Theorems \ref{homommzvstuffle} -- \ref{shufflestufflecombi}. It is believed (empirically and experimentally) that all $\Q$-linear relations between MZVs are deducible from these Theorems. This leads to the following standing conjecture.
\begin{conjecture*}
We have the equality
\begin{equation*}
    \ker(\zeta)=\{ u\ast w - u\shuffle w \mid u\in\Ha^1 \text{ and }v\in\Ha^0\}.
\end{equation*}
\end{conjecture*}
For $k\geq 1$, let us define the spaces 
\begin{equation*}
    Z_k=\big\langle\zeta(\s)\;\big|\; \s \text{ is admissible and }|\s|=k\big\rangle_\Q
\end{equation*}
and let us denote their dimensions by $d_k=\dim_\Q Z_k$. The first values are given by $d_1=0$, $d_2=1$ because $\zeta(2)=\frac{\pi^2}{6}\not\in\Q$, $d_3=1$ because $\zeta(2,1)=\zeta(3)\not\in\Q$ and $d_4=1$ due to the previous example and the fact that $\zeta(4)=\frac{\pi^4}{90}\not\in\Q$. At weight $k=5$ the dimension seems to be bounded by $d_5\leq2$, as we have 
\begin{equation*}
    Z_5\subseteq \Q\langle \zeta(5),\zeta(3)\zeta(2)\rangle.
\end{equation*}
The conjectured dimensions for larger $k$ are given in Table \ref{table:dimensions}.
\begin{table}[H]
\centering
\begin{tabular}{l|rrrrrrrrrrrrrrrr}
$k$ & 1& 2& 3& 4& 5& 6& 7& 8& 9& 10& 11& 12& 13& 14& 15&16\\
\hline
$d_k$ & 0 & 1& 1& 1& 2& 2& 3& 4& 5& 7& 9& 12& 16& 21& 28& 37
\end{tabular}
\caption{Conjectured dimensions $d_k$.}
\label{table:dimensions}
\end{table}
More generally, we have the following conjecture regarding these values.
\begin{conjecture*}
For $k\geq3$ we have the following recurrence:
\begin{equation*}
    d_k=d_{k-3}+d_{k-2}.
\end{equation*}
\end{conjecture*}

\subsection{Multiple harmonic sums}
\label{subsection:multiple harmonic sums}
A natural thing to do is to look at the direct multiplication of two MZVs. Consider the finite version of the multiple zeta values, given by
\begin{equation*}
    H(s_1,\dots,s_l;N)=H(\s;N)\coloneqq \sum_{N\geq n_1>\dots>n_l>0}\frac{1}{n_1^{s_1}\cdots n_l^{s_l}}.
\end{equation*}
These sums will be referred to as \textit{multiple harmonic sums}. Since the sums are finite, they are also defined for non-admissible indices. For a word $x_\s\in \Ha^1$ we will set $H(x_\s;N)=H(\s;N)$.
\begin{proposition}
\label{stufflemhs}
For any integer $N$ admissible words $v,w\in\Ha^1$ we have
\begin{equation*}
    H(v;N)\cdot H(w;N)=H(v\ast w;N).
\end{equation*}
\end{proposition}
\noindent\textit{Proof (Theorem \ref{homommzvstuffle}):}
If $v,w$ are admissible, the multiple harmonic sums converge to its corresponding multiple zeta values
\begin{equation*}
    \pushQED{\qed} 
\lim_{N\to\infty}H(s_1,\dots,s_l;N)=\zeta(s_1,\dots,s_l).\qedhere
\popQED
\end{equation*}

\noindent\textit{Proof (Proposition \ref{stufflemhs}):} We will apply induction on the sum of the lengths of $v$ and $w$. We may assume that the words $v$ and $w$ are non-empty, so let us consider the words $v=y_{s_1}y_{s_2}\cdots y_{s_l}$ and $w=y_{t_1}y_{t_2}\cdots y_{t_k}$. Note that we have the immediate relation
\begin{equation}
\label{mhs}
    H(s_1,\dots,s_l;N)=\sum_{n_1=1}^N\frac{1}{n_1^{s_1}}H(s_2,\dots,s_l;n_1-1).
\end{equation}
Now, if we split the product
\begin{align*}
    H(v;N)\cdot H(w;N)=H(s_1,\dots,s_l;N)\cdot H(t_1,\dots,t_k;N)&=\sum_{\substack{N\geq n_1>\dots>n_l>0\\N\geq m_1>\dots>m_l>0}}\frac{1}{n_1^{s_1}\cdots n_l^{s_l}}\;\frac{1}{m_1^{t_1}\cdots m_k^{t_k}}
\end{align*}
into the cases where $n_1>m_1$, $n_1<m_1$ and $n_1=m_1$, and apply relation (\ref{mhs}), we obtain
\begin{align*}
    H(s_1,\dots,s_l;N)H(t_1,\dots,t_k;N)&=\sum_{ n_1=1}^N\frac{1}{n_1^{s_1}}H(s_2,\dots,s_l;n_1-1)H(t_1,\dots,t_k;n_1-1)\\
    &\quad+\sum_{ m_1=1}^N\frac{1}{m_1^{t_1}}H(s_1,\dots,s_l;m_1-1)H(t_2,\dots,t_k;m_1-1)\\
    &\quad+\sum_{ m_1=1}^N\frac{1}{m_1^{s_1+t_1}}H(s_2,\dots,s_l;m_1-1)H(t_2,\dots,t_k;m_1-1).\\
\end{align*}
When we apply the induction hypothesis this becomes
\begin{align*}
    H(v;N)\cdot H(w;N)&=\sum_{ n_1=1}^N\frac{1}{n_1^{s_1}}H(y_{s_2}\cdots y_{s_l}\ast y_{t_1}\cdots y_{t_k};n_1-1)\\
    &\quad+\sum_{ m_1=1}^N\frac{1}{m_1^{t_1}}H(y_{s_1}\cdots y_{s_l}\ast y_{t_2}\cdots y_{t_k};m_1-1)\\
    &\quad+\sum_{ m_1=1}^N\frac{1}{m_1^{s_1+t_1}}H(y_{s_2}\cdots y_{s_l}\ast y_{t_2}\cdots y_{t_k};m_1-1).
\end{align*}
Finally, we apply (\ref{mhs}) again to obtain the required relation
\begin{align*}
    \pushQED{\qed} H(v;N)H(w;N)&=H(y_{s_1}(y_{s_2}\cdots y_{s_l}\ast y_{t_2}\cdots y_{t_k});N)+H(y_{s_1}\cdots y_{s_l}\ast y_{t_2}\cdots y_{t_k};N)\\
    &\quad+H(y_{s_1+t_1}(y_{s_2}\cdots y_{s_l}\ast y_{t_2}\cdots y_{t_k});N)\\
    &=H(v\ast w;N).\qedhere\popQED
\end{align*}

\subsection{Generalised polylogarithms}
\label{subsection:generalised polylogarithms}
To prove Theorem \ref{homommzvshuffle} we introduce the \textit{generalised polylogarithms}. For an arbitrary index $\s=(s_1,\dots,s_l)$ and a real $-1<z<1$, they are defined as
\begin{equation*}
    \Li_\s(z)=\sum_{n_1>\dots>n_l>0}\frac{z^{n_1}}{n_1^{s_1}\cdots n_l^{s_l}}.
\end{equation*}
Since we have the evaluation $\Li_\s(1)=\zeta(s)$, whenever $\s$ is admissible, the generalised polylogarithms can be seen as a generalisation of the multiple zeta values. As before, we set $\Li_{\mathbf{1}}(z)=1$ and $\Li_{x_\s}(z)=\Li_\s(z)$ for any admissible word, and extend this definition linearly to $\Ha^0$.\\

Let us consider the non-empty word $x_0v$, where $v$ is an arbitrary word $v\in\Ha^1$. When we differentiate with respect to $z$, we obtain
\begin{equation*}
    \frac{d}{dz}\Li_{x_0v}(z)=\sum_{n_1>\dots>n_l>0}\frac{n_1z^{n_1-1}}{n_1^{s_1}n_2^{s_2}\cdots n_l^{s_l}}=\frac{1}{z}\sum_{n_1>\dots>n_l>0}\frac{z^{n_1}}{n_1^{s_1-1}n_2^{s_2}\cdots n_l^{s_l}}=\frac{1}{z}\Li_v(z).
\end{equation*}
Similarly, when we regard the word $x_1v$, we have
\begin{equation*}
    \frac{d}{dz}\Li_{x_1v}(z)=\sum_{n_1>\dots>n_l>0}\frac{n_1z^{n_1-1}}{n_1n_2^{s_2}\cdots n_l^{s_l}}=\sum_{n_2>\dots>n_l>0}\sum_{n_1=0}^\infty z^{n_1}\frac{z^{n_2}}{n_2^{s_2}\cdots n_l^{s_l}}=\frac{1}{1-z}\Li_v(z).
\end{equation*}
This means that for simplicity, we can write $\frac{d}{dz}\Li_{x_iv}(z)=\omega_i(z)\Li_v(z)$, with 
\begin{equation*}
    \omega_{x_i}(z)=\omega_i(z)\coloneqq \begin{cases}\frac{1}{z}&\text{ if }i=0\\\frac{1}{1-z}&\text{ if }i=1 \end{cases}.
\end{equation*}
With this property between generalised polylogarithms, we are able to prove the following proposition.
\begin{proposition}
The map $\Li\colon \Ha^0_\shuffle\to C([0,1],\R)$ is a homomorphism of algebras, i.e. for any words $v,w\in\Ha$ we have
\begin{equation*}
    \Li_{v\shuffle w}(z)=\Li_v(z)\cdot\Li_w(z).
\end{equation*}
\end{proposition}
\proof
Let $x_iv$ and $x_jw$ be two words. We will prove the statement using induction on the sum of lengths $l(x_iv)+l(x_jw)$. Namely, we have
\begin{align*}
    \frac{d}{dz}\Li_{x_iv}(z)\Li_{x_jw}(z)&=\left(\frac{d}{dz}\Li_{x_iv}(z)\right)\Li_{x_jw}(z)+\Li_{x_iv}(z)\left(\frac{d}{dz}\Li_{x_jw}(z)\right)\\
    &=\omega_i(z)\Li_{v}(z)\Li_{x_jw}(z)+\omega_j(z)\Li_{x_iv}(z)\Li_w(z)\\
    &=\omega_i(z)\Li_{v\shuffle x_jw}(z)+\omega_j(z)\Li_{x_iv\shuffle w}(z)\\
    &=\frac{d}{dz}\left(\Li_{x_i(v\shuffle x_jw)}(z)+\Li_{x_j(x_iv\shuffle w)}(z)\right)\\
    &=\frac{d}{dz}\Li_{x_iv\shuffle x_jw}(z).
\end{align*}
This means that we have $\Li_{x_iv}(z)\Li_{x_jw}(z)=\Li_{x_iv\shuffle x_jw}(z)+C$ for some constant $C$. Substituting $z=0$ gives us $C=0$, which finishes the proof.\qed\\

\noindent\textit{Proof (Theorem \ref{homommzvshuffle}):} Let $v$ and $w$ be two admissible words. Then the previous proposition gives us
\begin{equation*}
    \pushQED{\qed} \zeta(v\shuffle w)=\lim_{z\to 1}\Li_{v\shuffle w}(z)=\lim_{z\to 1}\Li_{v}(z)\Li_{w}(z)=\zeta(v)\zeta(w).\qedhere\popQED
\end{equation*}

\noindent\textbf{Remark.}
Consider the shuffle product $x_0^{s-1}x_1\shuffle x_0^{t-1}x_1$. One way to evaluate this product is to compare the words $x_0^{s-1}x_1$ and $x_0^{t-1}x_1$ to decks of cards, consisting of $s-1$ and $t-1$ cards marked ``$0$'', followed by one card marked ``$1$'', respectively. Then every term appearing in the evaluation of $x_0^{s-1}x_1\shuffle x_0^{t-1}x_1$ corresponds to one possible way to riffle shuffle these two decks of cards (hence the name). Every shuffle is of the form
\begin{equation*}
    \underbrace{\vphantom{H^N_N}0\; 0 \;\cdots\; 0 \; 0}_{s-1 \text{ cards}}\; 1\hphantom{1}+\hphantom{1}\underbrace{\vphantom{H^N_N}0\; 0 \;\cdots\; 0 \; 0}_{t-1 \text{ cards}}\; 1 \hphantom{1}\xRightarrow{\text{shuffle}}\hphantom{1}
    \underbrace{\vphantom{H^N_N}0\; 0 \;\cdots\; 0 \; 0}_{j-1 \text{ cards}}\; \vertarrowbox{1}{middle ``$1$''-card}\; \underbrace{\vphantom{H^N_N}0\; 0 \;\cdots\; 0 \; 0}_{s+t-j-1 \text{ cards}}\; 1,
\end{equation*}
for some $j$. If we assume that the middle ``$1$''-card is the ``$1$''-card in the deck corresponding to $x_0^{s-1}x_1$, then the $j-1$ ``$0$''-cards consist of all $s-1$ ``$0$''-cards from the deck corresponding to $x_0^{s-1}x_1$ and the remaining $j-s+1$ ``$0$''-cards come from the deck corresponding to $x_0^{t-1}x_1$. And so there are ${j-1 \choose s-1}$ possible arrangements of these cards. Similarly, if the middle ``$1$''-card is the ``$1$''-card from the word $x_0^{t-1}x_1$, there are ${j-1\choose t-1}$ possible arrangements. We find the expansion
\begin{equation*}
    x_0^{s-1}x_1\shuffle x_0^{t-1}x_1 = \sum_{j=1}^{s+t-1}\left({j-1\choose t-1}+{j-1\choose s-1}\right)x_0^{j-1}x_1x_0^{s+t-j-1}x_1,
\end{equation*}
so that Proposition \ref{homommzvshuffle} follows. This can be made rigorous by induction on $s+t$.

\subsection{Connected sums}
\label{subsection:connected sums}
Finally, we will give a proof of Theorem \ref{dualitymzv}. Originally, Theorem \ref{dualitymzv} was proven by a change of variable in some integral representation of the multiple zeta values. We will however, follow a different proof, given by Seki and Yamamoto \cite{duality}. The proof uses the following identity.
\begin{lemma}
\label{factoriallemma}
For integers $n>0$ and $m\geq 0$ the identity 
\begin{equation*}
    \sum_{k=m+1}^\infty\frac{1}{k}\;\frac{k!n!}{(k+n)!}=\frac{1}{n}\;\frac{m!n!}{(m+n)!}
\end{equation*}
holds.
\end{lemma}
\begin{proof}
The statement follows directly by noting that the series is a telescoping series with 
\begin{equation*}
    \frac{1}{k}\;\frac{k!n!}{(k+n)!}=\frac{1}{n}\;\frac{(k-1)!n!}{(k+n-1)!}-\frac{1}{n}\;\frac{k!n!}{(k+n)!}.\qedhere
\end{equation*}
\end{proof}

\noindent \textit{Proof (Theorems \ref{duality} \& \ref{dualitymzv}).} For two multi-indices $\s=(s_1,\dots,s_l)$ and $\mathbf{t}=(t_1,\dots,t_p)$ we define the \textit{connected sum} to be the series
\begin{equation*}
    Z(x_\s,x_\mathbf{t})=Z(\s,\mathbf{t})\coloneqq\sum_{\substack{n_1>\dots>n_l>0\\m_1>\dots>m_p>0}}\frac{n_1!m_1!}{(n_1+m_1)!}\cdot\prod_{i=1}^l\frac{1}{n_i^{s_i}}\prod_{j=1}^p\frac{1}{m_j^{t_j}}.
\end{equation*}
Let $\s=(s_1,\dots,s_l)$ be an admissible index. We have the immediate equalities between the connected sums and the multiple zeta values $\zeta(\s)=\zeta(x_\s)=Z(x_\s,\varnothing)=Z(\varnothing, x_\s)$. We claim that 
\begin{equation}
\label{eq:connectedsum1}
    Z(s_1+1,s_2,\dots,s_l;t_1,\dots,t_p)=Z(s_1,\dots,s_l;1,t_1,\dots,t_p).
\end{equation}
The claim follows by applying Lemma \ref{factoriallemma} to the connected sum. Namely, we have
\begin{align*}
    Z(s_1+1,s_2,\dots,s_l;t_1,\dots,t_p)&=\sum_{m_1>\dots>m_p>0}\frac{1}{m_1^{t_1}\cdots m_p^{t_p}}\sum_{n_1>\dots>n_l>0}\frac{1}{n_1^{s_1}\cdots n_l^{s_l}}\;\frac{1}{n_1}\;\frac{n_1!m_1!}{(n_1+m_1)!}\\
    &=\sum_{m_1>\dots>m_p>0}\frac{1}{m_1^{t_1}\cdots m_p^{t_p}}\sum_{n_1>\dots>n_l>0}\frac{1}{n_1^{s_1}\cdots n_l^{s_l}}\sum_{k=m_1+1}^\infty\frac{1}{k}\;\frac{k!n_1!}{(k+n_1)!}\\
    &=\sum_{k>m_1>\dots>m_p>0}\frac{1}{km_1^{t_1}\cdots m_p^{t_p}}\sum_{n_1>\dots>n_l>0}\frac{1}{n_1^{s_1}\cdots n_l^{s_l}}\;\frac{k!n_1!}{(k+n_1)!}\\
    &=Z(s_1,\dots,s_l;1,t_1,\dots,t_p).
\end{align*}
This proves the claim. Then by symmetry we also have
\begin{equation}
\label{eq:connectedsum2}
    Z(1,s_1,\dots,s_l;t_1,\dots,t_p)=Z(s_1,\dots,s_l;t_1+1,\dots,t_p).
\end{equation}
The statement of the theorem follows by repeated use of equations (\ref{eq:connectedsum1}) and (\ref{eq:connectedsum2}). Namely, we have
\begin{align*}
    Z(\s;\varnothing)=Z(s_1,s_2\dots,s_l;\varnothing)&=Z(u_1+1,\underbrace{1,\dots,1}_{t_1-1\text{ times}},u_2+1,\underbrace{1,\dots,1}_{t_2-1\text{ times}},\dots,u_p+1,\underbrace{1,\dots,1}_{t_p-1\text{ times}};\varnothing)\\
    &=Z(u_1,\underbrace{1,\dots,1}_{t_1-1\text{ times}},u_2+1,\underbrace{1,\dots,1}_{t_2-1\text{ times}},\dots,u_p+1,\underbrace{1,\dots,1}_{t_p-1\text{ times}};1)\\
    &=Z(u_1-1,\underbrace{1,\dots,1}_{t_1-1\text{ times}},u_2+1,\underbrace{1,\dots,1}_{t_2-1\text{ times}},\dots,u_p+1,\underbrace{1,\dots,1}_{t_p-1\text{ times}};1,1)\\
    &\qquad\qquad\qquad\qquad\quad\vdots\\
    &=Z(1,\underbrace{1,\dots,1}_{t_1-1\text{ times}},u_2+1,\underbrace{1,\dots,1}_{t_2-1\text{ times}},\dots,u_p+1,\underbrace{1,\dots,1}_{t_p-1\text{ times}};1,\underbrace{1,\dots,1}_{u_1-1\text{times}})\\
    &=Z(\underbrace{1,\dots,1}_{t_1-1\text{ times}},u_2+1,\underbrace{1,\dots,1}_{t_2-1\text{ times}},\dots,u_p+1,\underbrace{1,\dots,1}_{t_p-1\text{ times}};2,\underbrace{1,\dots,1}_{u_1-1\text{times}})\\
    &\qquad\qquad\qquad\qquad\quad\vdots\\
    &=Z(u_2+1,\underbrace{1,\dots,1}_{t_2-1\text{ times}},\dots,u_p+1,\underbrace{1,\dots,1}_{t_p-1\text{ times}};t_1+1,\underbrace{1,\dots,1}_{u_1-1\text{times}})\\
    &=Z(u_2,\underbrace{1,\dots,1}_{t_2-1\text{ times}},\dots,u_p+1,\underbrace{1,\dots,1}_{t_p-1\text{ times}};1,t_1+1,\underbrace{1,\dots,1}_{u_1-1\text{times}})\\
    &\qquad\qquad\qquad\qquad\quad\vdots\\
    &=Z(\varnothing;t_k+1,\underbrace{1,\dots,1}_{u_k-1\text{ times}},t_{k-1}+1,\underbrace{1,\dots,1}_{u_{k-1}-1\text{ times}},\dots,t_1+1,\underbrace{1,\dots,1}_{u_1-1\text{ times}})\\
    &=Z(\varnothing;\s^\dagger).
\end{align*}


\noindent\textbf{Example.} The smallest non-trivial example is
\begin{equation*}
    \zeta(2,1)=Z(2,1;\varnothing)=Z(1,1;1)=Z(1;2)=Z(\varnothing;3)=\zeta(3).
\end{equation*}

\section{$q$-Analogues of the multiple zeta values}
When studying mathematics, it is sometimes rewarding to introduce another unknown. For example, generating functions arise this way. Other examples include the multiple harmonic sums and the polylogarithms from Sections \ref{subsection:multiple harmonic sums} and \ref{subsection:generalised polylogarithms}, respectively. This additional unknown may lead to more general identities, which could lead to a better understanding of the studied matter. In this chapter we will give different generalisations of the multiple zeta values by, in a helpful way, introducing the argument $0\leq q<1$. The idea is that, whenever we take the limit $q\to1^-$ of these generalisations, we return to multiple zeta values.
\subsection{$q$-Multiple zeta values}
\label{subsection:qmultiplezetavalues}
One way to generalise the multiple zeta values is through the limit
\begin{equation*}
    n=\lim_{q\to1^-}1+q+q^2+\dots+q^{n-1}=\lim_{q\to1^-}\frac{1-q^n}{1-q}.
\end{equation*}
For an integer $n>0$ we define the series
\begin{equation*}
    [n]_q=[n]=\frac{1-q^n}{1-q},
\end{equation*}
which in some sense generalises the integer $n$. For an admissible index $\s=(s_1,\dots,s_l)$ we define the $q$\textit{-multiple zeta value} to be the $q$-series
\begin{equation*}
    \zeta_q(\s)\coloneqq\sum_{n_1>\dots>n_l>0}\frac{q^{n_1(s_1-1)}\cdots q^{n_l(s_l-1)}}{[n_1]^{s_1}\cdots[n_l]^{s_l}}.
\end{equation*}
And so by construction, the $q$-multiple zeta value satisfies $\lim_{q\to1^-}\zeta_q(\s)=\zeta(\s)$. Note that we could have taken a different numerator when defining the $q$-multiple zeta value. In fact, we would have
\begin{equation*}
    \lim_{q\to1^-}\sum_{n_1>\dots>n_l>0}\frac{P(n_1,s_1;q)\cdots P(n_l,s_l;q)}{[n_1]^{s_1}\cdots[n_l]^{s_l}}=\zeta(\s)
\end{equation*}
for any polynomials $P(n_i,s_i;q)$ with $\lim_{q\to1^-}P(n_i,s_i;q)=1$. In Section \ref{section:monobrackets} we will consider a different $q$-analogue by considering different polynomials $P_(n,s)$.

\noindent\textbf{Example.} We have
\begin{align*}
    \zeta_q(3)&=q^4 + q^8 - 5q^9 + 15q^{10} - 35q^{11} + 71q^{12} - 131q^{13} + 220q^{14} +\dots,\\
    \zeta_q(4,1)&=q^2 + 2q^4 - 7q^5 + 19q^6 - 41q^7 + 70q^8 - 95q^9 + 111q^{10} + \dots\;.
\end{align*}

We will demonstrate the power of these $q$-generalisations by giving a proof of Ohno's relations (Theorem \ref{ohnosrelations}), by first proving some relation between $q$-series, and then take the limit $q\to 1^-$, which will give us the desired relation on the level of MZVs. The proof is due to S. Seki and S. Yamamoto, \cite{duality}, and uses a $q$-generalisation of the connected sums from Section \ref{subsection:connected sums}. For indices $\s=(s_1,\dots,s_l)$ and $\mathbf{t}=(t_1,\dots,t_p)$ we define the $q$-generalised connected sums as
\begin{align*}
    Z_q(\s;\mathbf{t};x)&\coloneqq \sum_{\substack{n_1>\dots>n_l>0\\m_1>\dots>m_p>0}}\frac{q^{n_1m_1}f_q(n_1;x)f_q(m_1;x)}{f_q(n_1+m_1;x)}\prod_{i=1}^l\frac{q^{(s_i-1)n_i}}{([n_i]-q^{n_i}x)[n_i]^{s_i-1}}\\
    &\qquad\cdot\prod_{j=1}^p\frac{q^{(t_j-1)m_j}}{([m_j]-q^{m_j}x)[m_j]^{t_j-1}}.
\end{align*}
Here, for an integer $n\geq0$, the series $f_q(n;x)$ are defined as
\begin{equation*}
    f_q(n;x)\coloneqq\prod_{j=1}^n([j]-q^j x).
\end{equation*}

When we evaluate $x=0$, we see that $Z_q(\s;\varnothing;0)=Z_q(\varnothing;\s;0)=\zeta_q(\s)$. Then, by taking the limit $q\to 1^-$, the series $Z_q(\s,\mathbf{t};0)$ becomes the connected sum from Section \ref{subsection:connected sums}. These $q$-connected sums also satisfy the usual duality.
\begin{theorem}
\label{dualityqx}
Let the multi-index $\s$ be admissible. The equality 
\begin{equation*}
    Z_q(\s;\varnothing;x)=Z_q(\varnothing;\s^\dagger;x)
\end{equation*}
holds.
\end{theorem}
\begin{proof}
The proof is similar to that of Theorem \ref{duality} and \ref{dualitymzv}. We again claim that
\begin{equation*}
    Z_q(s_1+1,s_2,\dots,s_l;t_1,\dots,t_p;x)=Z_q(s_1,\dots,s_l;1,t_1,\dots,t_p;x).
\end{equation*}
We have the following identity
\begin{align*}
    &\frac{q^{(k-1)n}f_q(k-1;x)f_q(n;x)}{f_q(k+n-1;x)}-\frac{q^{kn}f_q(k;x)f_q(n;x)}{f_q(k+n;x)}\\
    &\qquad=\frac{q^{(k-1)n}f_q(k-1;x)f_q(n;x)\big([k+n]-q^{k+n}x\big)-q^{(k-1)n}f_q(k-1;x)f_q(n;x)\big([k]-q^kx\big)}{f_q(k+n;x)}\\
    &\qquad=\frac{\big([k+n]-q^n[k]\big)q^{(k-1)n}f_q(k-1;x)f_q(n;x)}{f_q(k+n;x)}\\
    &\qquad=\frac{[n]q^{(k-1)n}f_q(k-1;x)f_q(n;x)}{f_q(k+n;x)}.
\end{align*}
This means that, analogously to Lemma \ref{factoriallemma}, we have the telescoping series
\begin{align*}
    &\sum_{k=m+1}^\infty \frac{1}{[k]-q^kx}\; \frac{q^{kn}f_q(k;x)f_q(n;x)}{f_q(k+n;x)}\\
    &\qquad=\sum_{k=m+1}^\infty \frac{q^n}{[n]}\left(\frac{q^{(k-1)n}f_q(k-1;x)f_q(n;x)}{f_q(k+n-1;x)}-\frac{q^{kn}f_q(k;x)f_q(n;x)}{f_q(k+n;x)} \right)\\
    &\qquad=\frac{q^n}{[n]}\;\frac{q^{mn}f_q(m;x)f_q(n;x)}{f_q(m+n;x)}.
\end{align*}
By symmetry, we also have the relation
\begin{equation*}
    Z_q(1,s_1,\dots,s_l;t_1,\dots,t_p;x)=Z_q(s_1,\dots,s_l;t_1+1,\dots,t_p;x).
\end{equation*}
The statement then follows by repeated use of these two identities.
\end{proof}
\begin{theorem}
\label{qsumeduality}
For an admissible index $\s=(s_1,\dots,s_l)$ we define the sum
\begin{equation*}
    S_q(\s;c)\coloneqq \sum_{\substack{c_1,\dots,c_l\geq 0\\c_1+\dots+c_l=c}}\zeta_q(s_1+c_1,\dots,s_l+c_l).
\end{equation*}
These quantities then satisfy the equality
\begin{equation*}
    S_q(\s;c)=S_q(\s^\dagger;c).
\end{equation*}
\end{theorem}
\begin{proof}
The statement follows from the equality
\begin{align*}
    Z_q(\s,\varnothing;x)&=\sum_{n_1>\dots>n_l>0}\prod_{i=1}^l\frac{q^{(s_i-1)n_i}}{([n_i]-q^{n_i}x)[n_i]^{s_i-1}}\\
    &=\sum_{n_1>\dots>n_l>0}\prod_{i=1}^l\frac{q^{(s_i-1)n_i}}{[n_i]^{s_i-1}}\sum_{c_i=0}^\infty \frac{q^{c_in_i}x^{c_i}}{[n_i]^{c_i}}\\
    &=\sum_{n_1>\dots>n_l>0}\prod_{i=1}^l\sum_{c_i=0}^\infty \frac{q^{(c_i+s_i-1)n_i}}{[n_i]^{c_i+s_i-1}}x^{c_i}\\
    &=\sum_{c=0}^\infty\sum_{\substack{c_1,\dots,c_l\geq0\\c_1+\dots+c_l=c}}^\infty\prod_{i=1}^l\sum_{n_1>\dots>n_l>0} \frac{q^{(c_i+s_i-1)n_i}}{[n_i]^{c_i+s_i-1}}x^c=\sum_{c=0}^\infty S_q(\s;c)x^c
\end{align*}
and Theorem \ref{dualityqx}, by considering the coefficient at $x^c$.
\end{proof}

\noindent\textit{Proof (Theorem \ref{ohnosrelations}):} Ohno's relations follows by taking the limit $q\to1^-$ in the identity $S_q(\s;c)=S_q(\s^\dagger;c)$ from Theorem \ref{qsumeduality} .\qed\\

One requirement we could set for our $q$-generalisation is that it, similar to the MZV case, satisfies some stuffle or shuffle product. I.e., are there algebraic interpretations of the generalised multiple zeta values such that there are products $\ast_q$ and $\shuffle_q$ which satisfy 
\begin{equation*}
    \zeta_q(v\astq w)=\zeta_q(v\shuffle_q w)=\zeta_q(v)\zeta_q(w)\text{?}
\end{equation*}
Indeed, let us consider the alphabet $A=\{y_s\}_{s>0}$ and the $\Q$-algebra $A_q=\Q\langle q\rangle\langle A\rangle$ generated by $A$. On $A_q$ we inductively define the product $\ast_q$ by 
\begin{equation*}
    u\astq \mathbf{1}=\mathbf{1}\astq u=u,
\end{equation*}
for any letter $u$ and the empty letter $\mathbf{1}$, and 
\begin{equation*}
    y_sv\ast y_tw=y_s(v\ast y_tw)+y_t(y_sv\ast w)+\big(y_{s+t}+(1-q)y_{s+t-1}\big)(v\ast w),
\end{equation*}
for the words $y_sv$ and $y_tw$. Finally we extend this definition linearly over $A_q$. The product $\astq$ is commutative and associative, so that it defines the stuffle algebra $\mathfrak{A}_q=(A_q,\astq)$. The map $\zeta_q$ acts on it by
\begin{equation*}
    \zeta_q(y_{s_1}y_{s_2}\cdots y_{s_l})=\zeta_q(s_1,s_2,\dots,s_l).
\end{equation*}
Additionally, we set $\zeta_q(qv)=q\zeta_q(v)$.
\begin{lemma}
The map $\zeta_q\colon\mathfrak{A}_q\to C([0,1),\R)$ is a homomorphism of algebras, i.e. for any words $v,w\in\mathfrak{A}_q$ we have 
\begin{equation*}
    \zeta_q(v\astq w)=\zeta_q(v)\cdot \zeta_q(w).
\end{equation*}
\end{lemma}
\begin{proof}
If one of the words $v,w$ is empty there is nothing to prove. For integers $n,s,t>0$ we have
\begin{align*}
    (1-q)\frac{q^{n(s+r-2)}}{[n]^{s+r-1}}+\frac{q^{n(s+r-1)}}{[n]^{s+r}}&=\frac{(1-q)^{s+r}q^{n(s+r-2)}(1-q^n)}{(1-q^n)^{s+r-1}(1-q^n)}+\frac{q^{n(s+r-1)}(1-q)^{s+r}}{(1-q^n)^{s+r}}\\
    &=\frac{(1-q)^{s+r}[q^{n(s+r-2)}-(1-q)^{n(s+r-1)}+(1-q)^{n(s+r-1)}]}{(1-q^n)^{s+r}}\\
    &=\frac{q^{n(s+r-2)}(1-q)^{s+r-1}}{(1-q^n)^{r+s-1}}\\
    &=\frac{q^{n(s-1)}}{[n]^s}\;\frac{q^{n(r-1)}}{[n]^r}.
\end{align*}
Now suppose we have the words $v=y_{s_1}u=y_{s_1}y_{s_2}\cdots y_{s_l}$ and $w=y_{t_1}u'=y_{t_1}y_{t_2}\cdots y_{t_k}$. With the equation above we are able to calculate
\begin{align*}
    \zeta_q(v)\zeta_q(w)&=\sum_{\substack{n_1>\dots>n_l>0 \\ m_1>\dots>m_k>0}}\prod_{i=1}^l\frac{q^{n_i(s_i-1)}}{[n_i]^{s_i}}\prod_{j=1}^k\frac{q^{m_j(r_j-1)}}{[m_j]^{r_j}}\\
    &=\sum_{\substack{n_1>\dots>n_l>0 \\ m_1>\dots>m_k>0\\n_1<m_1}}\cdots \;\;+\sum_{\substack{n_1>\dots>n_l>0 \\ m_1>\dots>m_k>0\\m_1<n_1}}\cdots\;\;+\sum_{\substack{n_1>\dots>n_l>0 \\ m_1>\dots>m_k>0\\n_1=m_1}}\cdots\\
    &=\zeta_q(y_{s_1}(u\astq y_{t_1}u'))+\zeta_q(y_{t_1}(y_{s_1}u\astq u'))+(1-q)\zeta_q(y_{s_1+t_1-1}(u\ast u'))\\
    &\quad+\zeta_q(y_{s_1+t_1}(u\astq u'))\\
    &=\zeta_q(v\astq w),
\end{align*}
which was to be shown.
\end{proof}
Despite the fact that the $q$-multiple zeta values satisfy a stuffle relation, there is no known shuffle relation among these $q$-analogues.

\subsection{Mono-brackets}
\label{section:monobrackets}
In \cite{bachmon} Bachmann defined another $q$-analogue of the multiple zeta values. For integers $s_1,\dots,s_l>0$ we define the $q$-power series
\begin{equation*}
    \bbracket{s_1,\dots,s_l}=\frac{1}{(s_1-1)!\cdots(s_l-1)!}\sum_{\substack{n_1>\cdots>n_l>0\\d_1,\dots,d_l>0}}d_1^{s_1-1}\cdots d_l^{s_l-1}q^{n_1d_1+\dots+n_ld_l}.
\end{equation*}
These $q$-power series are called mono-brackets, and the $\Q$-algebra they span is denoted by $\MD$. For a bracket $\bbracket{s_1,\dots,s_l}$ we define its weight as the sum $s_1+\dots+s_l$ and its length as the integer $l$. The algebra $\MD$ can then be graded by either weight or length.\\

\noindent\textbf{Example.} Some of the smaller mono-brackets are given by:
\begin{align*}
    \bbracket{1}&=q+2q^2+2q^3+3q^4+2q^5+4q^6+2q^7+\dots,\\
    \bbracket{2}&=q+3q^2+4q^3+7q^4+6q^5+12q^6+8q^7+\dots,\\
    \bbracket{2,1}&=q^3+2q^4+6q^5+7q^6+15q^8+18q^8+25q^9+\dots,\\
    \bbracket{4,3,2,1}&=\frac{1}{12}\big(q^{10} + 2q^{11} + 6q^{12} + 15q^{13} + 39q^{14} + 63q^{15} + \dots\big).
\end{align*}
Note that, for a fixed integer $s>0$, we have
\begin{equation*}
    \sum_{d>0}d^{s-1}x^{d}=\left(x\frac{d}{dx}\right)^{s-1}\frac{x}{1-x}=\frac{P_s(x)}{(1-x)^s},
\end{equation*}
for some polynomial $P_s(x)$ depending only on $s$. This means that we can rewrite the mono-brackets as
\begin{align*}
    \bbracket{s_1,\dots,s_l}=\frac{1}{(s_1-1)!\cdots(s_l-1)!}\sum_{n_1>\cdots>n_l>0}\frac{P_{s_1}(q^{n_1})\cdots P_{s_l}(q^{n_l})}{(1-q^{n_1})^{s_1}\cdots(1-q^{n_l})^{s_l}}.
\end{align*}
The polynomials $P_s(x)$ satisfy the relation
\begin{align*}
    \frac{P_{s+1}(x)}{(1-x)^{s+1}}=x\frac{d}{dx}\frac{P_s(x)}{(1-x)^s}=\frac{P'_s(x)(1-x)x+sP_s(x)x}{(1-x)^{s+1}},
\end{align*}
so that we have the inductive relation $P_{s+1}(x)=P'_s(x)(1-x)x+sP_s(x)x$. Because we have the polynomial identity $P_1(x)\equiv1$, we find that $P_{s+1}(1)=sP_{s}(1)=s!$. We see that the mono-brackets satisfy
\begin{equation}
\label{eq:limitmonobrackets}
    \lim_{q\to 1^-}(1-q)^{s_1+\dots+s_l}\bbracket{s_1,\dots,s_l}=\zeta(s_1,\dots,s_l),
\end{equation}
whenever $s_1>1$, so that they are indeed a $q$-generalisation of the multiple zeta values.

\subsection{The algebra of mono-brackets}
\label{subsection:algebraofmonobrackets}
The multiple zeta values satisfy both stuffle and shuffle relations. It is a natural question if the mono-brackets also satisfy some type of stuffle or shuffle relation.\\

Recall the polylogarithms from Section \ref{subsection:generalised polylogarithms}, which are defined by
\begin{equation*}
    \Li_{s}(z)=\sum_{n>0}\frac{z^n}{n^s}.
\end{equation*}
If $s<1$, the polylogarithm $\Li_{1-s}(z)$ becomes the rational function $P_s(z)/(1-z)^s$. This means that we can write the mono-brackets as
\begin{equation*}
    \bbracket{s_1,\dots,s_l}=\sum_{n_1>\dots>n_l>0}\frac{\Li_{1-s_1}(q^{n_1})\cdots\Li_{1-s_l}(q^{n_l})}{(s_1-1)!\cdots(s_l-1)!}.
\end{equation*}
\begin{lemma}
\label{homomorphismmonobracketsl=1}
The product of two mono-brackets of length $1$ is given by
\begin{align*}
    \bbracket{s}\cdot\bbracket{t}&=\bbracket{s,t}+\bbracket{t,s}+\bbracket{s+t}+\sum_{j=1}^{s}\lambda^j_{s,t}\bbracket{j}+\sum_{j=1}^t\lambda_{t,s}^j\bbracket{j},
\end{align*}
where the coefficients $\lambda^j_{s,t}$ and $\lambda^j_{t,s}$ are given by
\begin{equation*}
    \lambda^j_{a,b}\coloneqq(-1)^{b-1}{a+b-j-1\choose a-j}\frac{B_{a+b-j}}{(a+b-j)!}.
\end{equation*}
Here, the numbers $B_k$ are the Bernoulli numbers.
\end{lemma}
\begin{proof}
We have 
\begin{align*}
    \bbracket{s}\cdot\bbracket{t}&=\sum_{n,m>0}\frac{\Li_{1-s}(q^n)\Li_{1-t}(q^m)}{(s-1)!(t-1)!}\\
    &=\sum_{n>m>0}\cdots\;+\sum_{m>n>0}\cdots\;+\sum_{n=m>0}\cdots\\
    &=\bbracket{s,t}+\bbracket{t,s}+\sum_{n>0}\frac{\Li_{1-s}(q^n)\Li_{1-t}(q^n)}{(s-1)!(t-1)!}.
\end{align*}
We are left to show, with $z=q^n$, that 
\begin{equation}
    \frac{\Li_{1-s}(z)\Li_{1-t}(z)}{(s-1)!(t-1)!}=\frac{\Li_{1-s-t}(z)}{(s+t-1)!}+\sum_{j=1}^{s}\lambda_{s,t}^j\frac{\Li_{1-j}(z)}{(j-1)!}+\sum_{j=1}^t\lambda_{t,s}^j\frac{\Li_{1-j}(z)}{(j-1)!}.\label{eq:generatingLXLY}
\end{equation}
Consider the generating function of the polylogarithms
\begin{align*}
    L(X)&=\sum_{k\geq0}\frac{\Li_{-k}(z)}{k!}X^k=\sum_{k\geq 0}\sum_{n>0}\frac{n^k}{z^nk!}X^k=\sum_{n>0}z^n\sum_{k\geq0}\frac{(nX)^{k}}{k!}\\
    &=\sum_{n>0}z^ne^{nX}=\sum_{n>0}(ze^X)^n=\frac{ze^X}{1-ze^X}.
\end{align*}
A direct calculation shows that
\begin{align*}
    L(X)L(Y)=\frac{z^2e^{X+Y}}{(1-ze^X)(1-ze^Y)}=\frac{1}{e^{X-Y}-1}L(X)+\frac{1}{e^{Y-X}-1}L(Y).
\end{align*}
Then, by definition of the Bernoulli numbers, we have
\begin{align}
\label{eq:emachtbernoulli}
    L(X)L(Y)&=\sum_{n\geq0}\frac{B_n}{n!}(X-Y)^{n-1}L(X)+\sum_{n\geq0}\frac{B_n}{n!}(Y-X)^{n-1}L(Y)\nonumber\\
    &=\sum_{n>0}\frac{B_n}{n!}(X-Y)^{n-1}L(X)+\sum_{n>0}\frac{B_n}{n!}(Y-X)^{n-1}L(Y)+\frac{L(X)-L(Y)}{X-Y}.
\end{align}
We are interested in the value of $\frac{\Li_{1-s}(z)\Li_{1-t}(z)}{(s-1)!(t-1)!}$, so we want to know the coefficient of $X^{s-1}Y^{t-1}$ in the equation above. First, note that we have
\begin{equation*}
    \frac{L(X)-L(Y)}{X-Y}=\sum_{k\geq0}\frac{\Li_{-k}(z)}{k!}\frac{X^k-Y^k}{X-Y}=\sum_{k\geq0}\frac{\Li_{-k}(z)}{k!}\sum_{j=0}^{k-1}X^jY^{k-1-j},
\end{equation*}
so that the coefficient of $X^{s-1}Y^{t-1}$ equals $\Li_{1-s-t}(z)/(s+t-1)!$. Secondly, we have
\begin{align*}
    \frac{B_n}{n!}(X-Y)^{n-1}L(X)&=\sum_{k\geq0}\sum_{j=0}^{n-1}(-1)^{n-j}\frac{B_n}{n!}{n-1\choose j}\frac{\Li_{-k}(z)}{k!}X^{j+k}Y^{n-1-j},
\end{align*}
whose coefficient of $X^{s-1}Y^{t-1}$ is
\begin{equation*}
    \sum_{j=0}^{s-1}(-1)^t\frac{B_{t+j}}{(t+j)!}{t+j-1\choose j}\frac{\Li_{1-s+j}(z)}{(s-j-1)!}.
\end{equation*}
After the suitable change of variables $j\mapsto s-1-j$, this becomes precisely $\lambda^j_{s,t}\frac{\Li_{1-j}(z)}{(j-1)!}$. The argument for the coefficient of $\frac{B_n}{n!}(X-Y)^{n-1}L(Y)$ is analogous. This completes the proof of (\ref{eq:generatingLXLY}), and finishes the proof of Lemma \ref{homomorphismmonobracketsl=1}.
\end{proof}

\noindent\textbf{Example.} For $s=t=2$ the identity in Lemma \ref{homomorphismmonobracketsl=1} becomes
\begin{align*}
    \bbracket{2}\cdot\bbracket{2}=2\bbracket{2,2}+\bbracket{4}-\frac{1}{6}\bbracket{2}.
\end{align*}
Note that when we apply the limit as in (\ref{eq:limitmonobrackets}), this identity becomes the stuffle relation between MZVs:
\begin{equation*}
    \zeta(2)\zeta(2)=2\zeta(2,2)+\zeta(4).
\end{equation*}
In fact, if $s,t>1$ we see that the weight of the terms $\lambda^j_{s,t}\bbracket{j}$ is always less then $s+t$. This means that when we apply the concerning limit, these terms vanish, and we obtain the stuffle relation
\begin{equation*}
    \zeta(s)\zeta(t)=\zeta(s,t)+\zeta(t,s)+\zeta(s+t).
\end{equation*}
The product in Lemma \ref{homomorphismmonobracketsl=1} can therefore be seen as the bracket version of the stuffle product in length $1$. This suggests that there exists a stuffle product for mono-brackets of all lengths. Indeed this is the case. Consider the alphabet $A=\{z_s\}_{s>0}$. For two letters $z_s,z_t\in A$ we define the product
\begin{equation*}
    z_s\circ z_t= z_{s+t}+\sum_{j=1}^s\lambda_{s,t}^jz_j+\sum_{j=1}^t\lambda_{t,s}^jz_j.
\end{equation*}
For two words in $\Q\langle A\rangle$ we inductively define the stuffle product to be $u\ast \mathbf{1}=\mathbf{1}\ast u=u$ if one of the words is empty, and
\begin{align*}
    z_sv\ast z_tw= z_s(v\ast z_tw)+z_t(z_sv\ast w)+(z_s\circ z_t)(v\ast w)
\end{align*}
otherwise. We extend this definition linearly, so that it defines the stuffle algebra $\Ha^\MD_\ast=(\Q\langle A\rangle,\ast)$. Since the product $\circ$ is symmetric, an inductive argument shows that the stuffle product $\ast$ is commutative and associative.
\begin{proposition}
\label{homomorphismmonobrackets}
The map $\bbracket{\,\cdot\,}\colon \Ha_\ast^\MD\to \MD$ is a homomorphism of $\Q$-algebras. I.e., we have
\begin{equation*}
    \bbracket{v \ast w}=\bbracket{v}\cdot \bbracket{w}.
\end{equation*}
\end{proposition}
\begin{proof}
The proof mimics the proof of Theorem \ref{homommzvstuffle}. Define the harmonic mono-bracket
\begin{equation*}
    \bbracket{z_{s_1}\cdots z_{s_l}}_N=\bbracket{s_1,\dots,s_l}_N\coloneqq\sum_{N\geq n_1>\dots>n_l>0}\frac{\Li_{1-s_1}(q^{n_1})\cdots\Li_{1-s_l}(q^{n_l})}{(s_1-1)!\cdots(s_l-1)!}
\end{equation*}
to be the finite, truncated version of the mono-bracket. Then by definition we have
\begin{equation*}
    \bbracket{z_{s_1}\cdots z_{s_l}}_N=\sum_{N\geq n_1>0}\frac{\Li_{1-s_1}(q^{n_1})}{(s_1-1)!}\bbracket{s_2,\dots,s_l}_{n_1-1},
\end{equation*}
and so the product of two harmonic mono-brackets becomes
\begin{align*}
    \bbracket{z_{s_1}\cdots z_{s_l}}_N\bbracket{z_{t_1}\cdots z_{t_k}}_N    &=\sum_{N\geq n_1,m_1>0}\frac{\Li_{1-s_1}(q^{n_1})\Li_{1-t_1}(q^{m_1})}{(s_1-1)!(t_1-1)!}\bbracket{s_2,\dots,s_l}_{n_1-1}\bbracket{t_2,\dots,t_k}_{m_1-1}\\
    &=\sum_{N\geq n_1>0} \frac{\Li_{1-s_1}(q^{n_1})}{(s_1-1)!} \bbracket{s_2,\dots,s_l}_{n_1-1}\bbracket{t_1,\dots,t_k}_{n_1-1} \\
    &+\sum_{N\geq m_1>0}\frac{\Li_{1-t_1}(q^{m_1})}{(t_1-1)!}\bbracket{s_1,\dots,s_l}_{m_1-1}\bbracket{t_2,\dots,t_k}_{m_1-1}\\
    &+\sum_{N\geq m_1>0} \frac{\Li_{1-s_1}(q^{m_1})\Li_{1-t_1}(q^{m_1})}{(s_1-1)!(t_1-1)!}\bbracket{s_2,\dots,s_l}_{m_1-1}\bbracket{t_2,\dots,t_k}_{m_1-1}.
\end{align*}
We will apply induction to the sum of lengths $l+k$. Then, together with equation (\ref{eq:generatingLXLY}), we obtain
\begin{align*}
    \bbracket{z_{s_1}\cdots z_{s_l}}_N\bbracket{z_{t_1}\cdots z_{t_k}}_N &=\sum_{N\geq n_1>0} \frac{\Li_{1-s_1}(q^{n_1})}{(s_1-1)!} \bbracket{z_{s_2},\dots,z_{s_l}\ast z_{t_1},\dots,z_{t_k}}_{n_1-1} \\
    &+\sum_{N\geq m_1>0}\frac{\Li_{1-t_1}(q^{m_1})}{(t_1-1)!}\bbracket{z_{s_1},\dots,z_{s_l}\ast z_{t_2},\dots,z_{t_k}}_{m_1-1}\\
    &+\sum_{N\geq m_1>0} \frac{\Li_{1-s_1}(q^{m_1})\Li_{1-t_1}(q^{m_1})}{(s_1-1)!(t_1-1)!}\bbracket{z_{s_2},\dots,z_{s_l}\ast z_{t_2},\dots,z_{t_k}}_{m_1-1}\\
    &=\bbracket{z_{s_1}(z_{s_2},\dots,z_{s_l}\ast z_{t_1},\dots,z_{t_k})}_N+ \bbracket{z_{t_1}(z_{s_1},\dots,z_{s_l}\ast z_{t_2},\dots,z_{t_k})}_N\\
    &+\bbracket{(z_{s_1}\circ z_{s_2})(z_{s_2},\dots,z_{s_l}\ast z_{t_2},\dots,z_{t_k})}_N\\
    &=\bbracket{z_{s_1}\cdots z_{s_l}\ast z_{t_1}\cdots z_{t_k}}_N.
\end{align*}
Finally, the statement of the proposition follows from taking the limit
\begin{align*}
    \bbracket{z_{s_1}\cdots z_{s_l}}\bbracket{z_{t_1}\cdots z_{t_k}}&=\lim_{N\to\infty}\bbracket{z_{s_1}\cdots z_{s_l}}_N\bbracket{z_{t_1}\cdots z_{t_k}}_N\\
    &=\lim_{N\to\infty}\bbracket{z_{s_1}\cdots z_{s_l}\ast z_{t_1}\cdots z_{t_k}}_N\\
    &=\bbracket{z_{s_1}\cdots z_{s_l}\ast z_{t_1}\cdots z_{t_k}}.\qedhere
\end{align*}
\end{proof}

We call a word $z_{s_1}z_{s_2}\cdots z_{s_l}$ admissible if it satisfies $s_1>1$. Furthermore, we define the subspace $\Ha_\ast^{\MD,0}\subseteq\Ha_\ast^\MD$ to be the space of all $\Q$-linear combinations of admissible words. Let $\qMZ\subseteq \MD$ be the image of $\Ha_\ast^{\MD,0}$ under $\bbracket{\,\cdot\,}$, i.e., it is the space of all $\Q$-linear combinations of mono-brackets $\bbracket{s_1,\dots,s_l}$ with $s_1>1$.
\begin{proposition}
\label{mbsubspaces}
The space $\Ha_\ast^{\MD,0}$ is a subalgebra of $\Ha_\ast^\MD$, so that $\qMZ$ is a subalgebra of $\MD$.
\end{proposition}
\begin{proof}
We will show that $\Ha_\ast^{\MD,0}$ is closed under multiplication. Consider the admissible words $z_{s}v$ and $z_{t}w$. Then by Proposition \ref{homomorphismmonobrackets} we have
\begin{equation*}
    z_{s}v\ast z_{t}w=z_{s}(v\ast z_{t}w)+z_{t}(z_{s}v\ast w)+(z_{s}\circ z_{t})(v\ast w).
\end{equation*}
Because $s,t>1$, we see that the first two terms are admissible. Additionally, the first letter of the last term is
\begin{align*}
    z_{s}\circ z_{t}=\sum_{j=1}^s \lambda_{s,t}^j z_j+\sum_{j=1}^t\lambda_{t,s}^jz_j +z_{s+t}.
\end{align*}
The only non-admissible words are the words corresponding to $j=1$. The coefficient of this word is 
\begin{align*}
    \lambda_{s,t}^1+\lambda_{t,s}^1&=(-1)^{t-1}{s+t-2\choose s-1}\frac{B_{s+t-1}}{(s+t-1)!}+(-1)^{s-1}{s+t-2\choose s-1}\frac{B_{s+t-1}}{(s+t-1)!}.
\end{align*}
If $s$ and $t$ differ in parity, both terms cancel. On the other hand, if $s$ and $t$ are either both odd or both even, then $s+t-1$ is odd, so that the Bernoulli number $B_{s+t-1}$ vanishes. In either case we have $\lambda_{s,t}^1+\lambda_{t,s}^1=0$, so that $(z_s\circ z_t)(v\ast w)$ is a linear combination of admissible words.
\end{proof}


\begin{proposition}
\label{mbtomzv}
For two admissible words $v,w$ we have
\begin{equation*}
    \lim_{q\to1}(1-q)^{|v|+|w|}\bbracket{v\ast w}=\zeta(v\ast w),
\end{equation*}
where ``$v\ast w$" on the left-hand side denotes the stuffle product between mono-brackets, and on the right-hand side between multiple zeta values.
\end{proposition}
\begin{proof}
Write $v=z_su$ and $w=z_tu'$. All terms with weight less than $|s|+|t|$ vanish under this limit. This means that we have
\begin{align*}
    \lim_{q\to1}(1-q)^{|v|+|w|}\bbracket{v\ast w}&=\lim_{q\to1}(1-q)^{|v|+|w|}\big(\bbracket{z_s(u\ast z_tu')}+\bbracket{z_t(z_su\ast u')}+\bbracket{z_{t+s}(u\ast u')}\big)\\
    &=\zeta(v\ast w).\qedhere
\end{align*}
\end{proof}

Though being equipped with a stuffle product, the space $\MD$ satisfies no apparent shuffle product. The lack of a shuffle relation between mono-brackets leads to the introduction of the so-called bi-brackets, which are introduced in Chapter \ref{chapter:bibrackets}.

\subsection{The operator $q\frac{d}{dq}$}
\label{subsection:derivations}
In this section we introduce the operator $d_q=q\frac{d}{dq}$. It turns out that this operator preserves the spaces $\MD$ and $\qMZ$, so that it is a derivation on these spaces. For example, consider the derivative 
\begin{equation*}
    d_q\negphantom{.}\bbracket{2}=d_q\sum_{n>0}\sum_{k>0}kq^{nk}=\sum_{n,k>0}nk^2q^{nk}.
\end{equation*}
It is not immediately clear that this series is again a linear combination of mono-brackets. The following result is due to Bachmann and K\"uhn \cite[Thm. 1.7, p 5]{bachmon}.
\begin{theorem}
\label{invarianceofd}
The operator $d_q=q\frac{d}{dq}$ is a self-map of $\MD$ and $\qMZ$. I.e., $d_q\negphantom{.}\bbracket{s_1,\dots,s_l}$ can be written as a linear combination of mono-brackets. Furthermore, $d_q\negphantom{.}\bbracket{s_1,\dots,s_l}$ is admissible whenever $\bbracket{s_1,\dots,s_l}$ is admissible.
\end{theorem}
In the proof of Theorem \ref{invarianceofd} we make use of the generating functions of the mono-brackets
\begin{equation*}
    T(X_1,\dots,X_l)=\sum_{s_1,\dots,s_l>0}[s_1,\dots,s_l]X_1^{s_1-1}\cdots X_l^{s_l-1}.
\end{equation*}
We can rewrite these generating functions as stated in the following lemma.

\begin{lemma}
\label{generatingmono}
We have
\begin{equation*}
    T(X_1,\dots,X_l)=\sum_{n_1,\dots,n_l>0}\prod_{j=1}^l\frac{e^{n_jX_j}q^{n_1+\dots+n_j}}{1-q^{n_1+\dots+n_j}}.
\end{equation*}
\end{lemma}
\begin{proof}
By definition we have
\begin{align*}
    T(X_1,\dots,X_l)&=\sum_{s_1,\dots,s_l>0}\left(\sum_{\substack{u_1>\cdots>u_l>0\\n_1,\dots,n_1>0}}\frac{n_1^{s_1-1}\cdots n_l^{s_l-1}}{(s_1-1)!\cdots(s_l-1)!}q^{u_1n_1+\dots+u_ln_l}\right)X_1^{s_1-1}\cdots X_l^{s_l-1}\\
    &=\sum_{s_1,\dots,s_l\geq0}\left(\sum_{\substack{u_1>\cdots>u_l>0\\n_1,\dots,n_1>0}}\frac{n_1^{s_1}\cdots n_l^{s_l}}{s_1!\cdots s_l!}q^{u_1n_1+\dots+u_ln_l}\right)X_1^{s_1}\cdots X_l^{s_l}.
\end{align*}
The statement $u_1>u_2>\dots>u_{l-1}>u_l>0$ means that for every $1\leq k\leq l-1$ there exists an integer $v_k$ such that $u_k=u_{k+1}+v_k$. When we set $v_l=u_l$, we obtain
\begin{alignat*}{3}
    &u_l&&&&\quad=v_l,\\
    &u_{l-1}&&=u_l+v_{l-1}&&\quad=v_l+v_{l-1,}\\
    &u_{l-2}&&=u_{l-1}+v_{l-2}&&\quad=v_l+v_{l-1}+v_{l-2},\\
    &&\vdots\\
    &u_1&&=u_2+v_1&&\quad=v_l+\dots+v_2+v_1.
\end{alignat*}
When we apply this transformation, our generating function becomes
\begin{align*}
    T(X_1,\dots,X_l)&=\sum_{s_1,\dots,s_l\geq0}\left(\sum_{\substack{v_1,\cdots,v_l>0\\n_1,\dots,n_1>0}}\frac{n_1^{s_1}\cdots n_l^{s_l}}{s_1!\cdots s_l!}q^{(v_1+\dots+v_l)n_1+(v_2
    +\dots+v_l)n_2+\dots+v_ln_l}\right)X_1^{s_1}\cdots X_l^{s_l}\\
    &=\sum_{s_1,\dots,s_l\geq0}\left(\sum_{\substack{v_1,\cdots,v_l>0\\n_1,\dots,n_1>0}}\frac{n_1^{s_1}\cdots n_l^{s_l}}{s_1!\cdots s_l!}q^{v_l(n_1+\dots+n_l)+v_{l-1}(n_1+\dots+n_{l-1})+\dots+v_1n_1}\right)X_1^{s_1}\cdots X_l^{s_l}\\
    &=\sum_{n_1,\dots,n_l\geq0}\prod_{j=1}^l\sum_{s_j\geq0}\frac{n_j^{s_j}}{s_j!}X^{s_j}\sum_{v_j>0}q^{v_j(n_1+\dots+n_j)}\\
    &=\sum_{n_1,\dots,n_l\geq0}\prod_{j=1}^le^{n_jX_j}\frac{q^{n_1+\dots+n_j}}{1-q^{n_1+\dots+n_j}},
    \end{align*}
which was to be shown.
\end{proof}

We will illustrate the proof of Theorem \ref{invarianceofd} by proving the case $l=1$. Consider the product between generating functions
\begin{align*}
    T(X)\cdot T(Y)&= \sum_{n,m>0}\frac{e^{nX}q^n}{1-q^n}\;\frac{e^{mY}q^m}{1-q^m}.
\end{align*}
This sum can be split into the cases where $n>m$, $m>n$ and $n=m$. If $n>m$, then we have $n=m+k$ for some integer $k>0$. This gives us
\begin{align*}
    \sum_{n>m>0}\frac{e^{nX}q^n}{1-q^n}\;\frac{e^{mY}q^m}{1-q^m}=\sum_{m,k>0}\frac{e^{m(X+Y)}q^m}{1-q^m}\;\frac{e^{kX}q^{k+m}}{1-q^{k+m}}=T(X+Y,X).
\end{align*}
In a similar way we find that the case $m>n$ corresponds to $T(X+Y,Y)$. Finally, consider the case $m=n$. Here we will apply the identity $\left(\frac{q^n}{1-q^n}\right)^2= \frac{q^n}{(1-q^n)^2}-\frac{q^n}{1-q^n}$, which gives us
\begin{equation*}
    \sum_{n>0}e^{n(X+Y)}\left(\frac{q^n}{1-q^n}\right)^2=\sum_{n>0}\frac{e^{n(X+Y)}q^n}{(1-q^n)^2}-T(X+Y).
\end{equation*}
To introduce the operator $d_q=q\frac{d}{dq}$ we note that it acts on $T(X+Y)$ by
\begin{equation*}
    d_qT(X+Y)=\sum_{n>0}\frac{ne^{n(X+Y)}q^n}{(1-q^n)^2}=\frac{\del}{\del X}\sum_{n>0}\frac{e^{n(X+Y)}q^n}{(1-q^n)^2}.
\end{equation*}
Combining all parts gives us
\begin{align*}
    d_qT(X+Y)=\frac{\del}{\del X}T(X)T(Y)-\frac{\del}{\del X}T(X+Y,X)-\frac{\del}{\del X}T(X+Y,Y)+\frac{\del}{\del X}T(X+Y).
\end{align*}
The terms on the right-hand side are all (shifted) generating functions of mono-brackets. So when we consider the coefficient of $X^{a-1}Y^{b-1}$ with $a+b=s$ we find a representation of $d_q\negphantom{.}\bbracket{s}$ in terms of mono-brackets. Note that there are multiple choices for these integers $a,b$, which give different representations for $d_q\negphantom{.}\bbracket{s}$.\\

Now we give the proof of Theorem \ref{invarianceofd} for arbitrary length.
\begin{proof}[Proof (Theorem \ref{invarianceofd})]
Using Lemma \ref{generatingmono}, we are able to calculate the product
\begin{align*}
    T(X)T(X_1,\dots,X_l)&=\sum_{m,n_1,\dots,n_l>0}e^{mX+n_1X_1+\dots+n_lX_l}\frac{q^m}{1-q^m}\prod_{j=1}^l \frac{q^{n_1+\dots+n_j}}{1-q^{n_1+\dots+n_j}}.
\end{align*}
The right-hand side can be split according to the cases $n_1+\dots+n_k>m>n_1+\dots+n_{k-1}$ and $m=n_1+\dots+n_k$ for $1\leq k\leq l$, and where $m>n_1+\dots+n_l$.\\

First, consider the case where $m>n_1+\dots+n_l$. With $M=m-n_1-\dots-n_l$ we have
\begin{align*}
    \Sigma_1&\coloneqq \sum_{\substack{m,n_1,\dots,n_l>0\\m>n_1+\dots+n_l}}e^{mX+n_1X_1+\dots+n_lX_l}\frac{q^m}{1-q^m}\prod_{j=1}^l \frac{q^{n_1+\dots+n_j}}{1-q^{n_1+\dots+n_j}}\\
    &=\sum_{n_1,\dots,n_l,M>0}e^{(n_1+\dots+n_l+M)X+n_1X_1+\dots+n_lX_l}\frac{q^{n_1+\dots+n_l+M}}{1-q^{n_1+\dots+n_l+M}}\prod_{j=1}^l \frac{q^{n_1+\dots+n_j}}{1-q^{n_1+\dots+n_j}}\\
    &=T(X_1+X,\dots,X_l+X,X).
\end{align*}
Now consider the cases where $n_1+\dots+n_k>m>n_1+\dots+n_{k-1}$ for some $1\leq k\leq l$. Then there exist some integers $M,N$ such that $m=n_1+\dots+n_{k-1}+M$ and $n_1+\dots+n_k=m+N$. When we apply these transformations we see that
\begin{align*}
    mX+n_1X_1+\dots+n_lX_l&=(n_1+\dots+n_{k-1}+M)X+n_1X_1+\dots+n_{k-1}X_{k-1}\\
    &\quad+(M+N)X_k+n_{k+1}X_{k+1}+\dots+n_lX_l\\
    &=n_1(X_1+X)+\dots+n_{k-1}(X_{k-1}+X)+M(X_k+X)\\
    &\quad+NX_k+n_{k+1}X_{k+1}+\dots+n_lX_l,
\end{align*}
that $q^m=q^{n_1+\dots+n_{k-1}+M}$ and finally, that 
\begin{equation*}
    q^{n_1+\dots+n_{k-1}+n_k+n_{k+1}+\dots+n_i}=q^{n_1+\dots+n_{k-1}+M+N+n_{k+1}+\dots+n_i}
\end{equation*}
for $i\geq k+1$.

Combining these parts give us
\begin{align*}
    \Sigma_2&\coloneqq\sum_{k=1}^l \sum_{\substack{m,n_1,\dots,n_l>0\\m>n_1+\dots+n_{k-1}\\m<n_1+\dots+n_{k}}}e^{mX+n_1X_1+\dots+n_lX_l}\frac{q^m}{1-q^m}\prod_{j=1}^l \frac{q^{n_1+\dots+n_k}}{1-q^{n_1+\dots+n_k}}\\
    &=T(X_1+X,\dots,X_k+X,X_{k},\dots,X_l).
\end{align*}

Finally, consider the cases where $m=n_1+\dots+n_k$ for some $1\leq k\leq l$. 
Using the equation $(\frac{q^m}{1-q^m})^2=\frac{q^m}{(1-q^m)^2}-\frac{q^m}{1-q^m}$ we obtain
\begin{align*}
    \Sigma_3&\coloneqq\sum_{k=1}^l \sum_{\substack{m,n_1,\dots,n_l>0\\m=n_1+\dots+n_{k}}}e^{mX+n_1X_1+\dots+n_lX_l}\frac{q^m}{1-q^m}\prod_{j=1}^l \frac{q^{n_1+\dots+n_k}}{1-q^{n_1+\dots+n_k}}\\
    &=\sum_{k=1}^l \sum_{n_1,\dots,n_l>0}e^{n_1(X+X_1)+\dots+n_k(X+X_k)+n_{k+1}X_{k+1}+\dots+n_lX_l}\prod_{j=1}^l\left(\frac{q^{n_1+\dots+n_j}}{1-q^{n_1+\dots+n_j}}\right)^{1+\delta_{k,j}}\\
    &=\sum_{k=1}^l \sum_{n_1,\dots,n_l>0}e^{n_1(X+X_1)+\dots+n_k(X+X_k)+n_{k+1}X_{k+1}+\dots+n_lX_l}\prod_{j=1}^l\frac{q^{n_1+\dots+n_j}}{(1-q^{n_1+\dots+n_j})^{1+\delta_{k,j}}}\\
    &\quad-\sum_{k=1}^lT(X_1+X,\dots,X_{k}+X,X_{k+1},\dots,X_l)\\
    &\eqqcolon R_l+\Sigma_3.
\end{align*}
Define the operator $D(f)=\restr{\left(\frac{\del}{\del X}f\right)}{X=0}$. We claim that $D(R_l)=dT(X_1,\dots,X_l)$. Namely, we have
\begin{align*}
    d_qT(X_1,\dots,X_l)&=\sum_{n_1,\dots,n_l>0}q\frac{d}{dq}\left(\prod_{j=1}^le^{n_jX_j}\frac{q^{n_1+\dots+n_j}}{1-q^{n_1+\dots+n_j}}\right)\\
    &=\sum_{n_1,\dots,n_l>0}\sum_{k=1}^l(n_1+\dots+n_k)\prod_{j=1}^le^{n_jX_j}\frac{q^{n_1+\dots+n_j}}{(1-q^{n_1+\dots+n_j})^{\delta_{i,j}}}\\
    &=\sum_{k=1}^l\sum_{n_1,\dots,n_l>0}(n_1+\dots+n_k)e^{n_1X_1+\dots+n_lX_l}\prod_{j=1}^l\frac{q^{n_1+\dots+n_j}}{(1-q^{n_1+\dots+n_j})^{1+\delta_{k,j}}}\\
    &=D(R_l),
\end{align*}
where we used the equation
\begin{equation*}
    q\frac{d}{dq}\left(\frac{q^m}{1-q^m}\right)=\frac{mq^m}{(1-q^m)^2}.
\end{equation*}
This means that if we apply $D$ to the equation
\begin{equation*}
    R_l=T(X) T(X_1,\dots,X_l)-\Sigma_1-\Sigma_2-\Sigma_3,
\end{equation*}
the left-hand side becomes $d_qT(X_1,\dots,X_l)$. For the right-hand side we note that the terms in $D(\Sigma_1+\Sigma_2+\Sigma_3)$ can be expressed as a linear combination of (shifted) generating functions. The same holds for $T(X)T(X_1,\dots,X_l)$, which becomes $\bbracket{2}\cdot\hphantom{m}\negphantom{n} T(X_1,\dots,X_l)$. This means that, by looking at the coefficient of the term $X_1^{s_1-1}\cdots X^{s_l-1}$, $d_q\negphantom{.}\bbracket{s_1,\dots,s_l}$ can be written as a linear combination of mono-brackets.
\end{proof}
\section{Bi-brackets}
\label{chapter:bibrackets}
In the previous chapter we noticed that the mono-brackets form a generalisation of the multiple zeta values. These brackets satisfy a stuffle relation. In \cite{bachbi} Bachmann introduced an alternative $q$-model, called the bi-brackets. These new brackets satisfy, in a natural way, both stuffle and shuffle relations. This chapter is based on Bachmann's work on bi-brackets \cite[Chapters 3, 4, 5, pp. 4-13]{bachbi}.\\

For $r_1,\dots,r_l\geq 0$ and $s_1,\dots,s_l>0$ we define the corresponding \textit{bi-bracket} to be the $q$-series
\begin{equation*}
    \begin{bmatrix}s_1,\dots,s_l\\r_1,\dots,r_l\end{bmatrix}\coloneqq \sum_{\substack{u_1>\dots>u_l>0\\v_1,\dots,v_l>0}}\frac{u_1^{r_1}\cdots u_l^{r_l}}{r_1!\cdots r_l!}\;\frac{v_1^{s_1-1}\cdots v_l^{s_l-1}}{(s_1-1)!\cdots(s_l-1)!}q^{u_1v_1+\dots+u_lv_l}.
\end{equation*}
\noindent\textbf{Example.} The first terms of some smaller bi-brackets are given by
\begin{align*}
    \bbracket{1,1\\1,0}&=2q^3 + 5q^4 + 14q^5 + 20q^6 + 39q^7 + 52q^8 + 74q^9 + \dots,\\
    \bbracket{2,2\\1,1}&=2q^3 + 7q^4 + 26q^5 + 46q^6 + 108q^7 + 172q^8 + 274q^9 + \dots,\\
    \bbracket{2,1\\1,3}&=\frac{1}{6}\big(2q^3 + 5q^4 + 37q^5 + 50q^6 + 208q^7 + 306q^8 + 669q^9 + \dots\big).
\end{align*}

We have the following immediate relation between mono-brackets and bi-brackets
\begin{equation*}
    \begin{bmatrix}s_1,\dots,s_l\end{bmatrix}=\begin{bmatrix}s_1,\dots,s_l\\0,\dots,0\end{bmatrix},
\end{equation*}
and so every mono-bracket can be trivially seen as a bi-bracket.\\

The \textit{length} of a bi-bracket $v=\bbracket{s_1,\dots,s_l\\r_1,\dots,r_l}$ is defined as the integer $l$. The \textit{upper} and \textit{lower weight} are the sums $s_1+\dots+s_l$ and $r_1+\dots+r_l$, respectively. The \textit{total weight}, or simply the weight, is defined as the upper weight plus the lower weight, and is denoted by $|v|$.

\subsection{The partition relation for bi-brackets}
We denote the generating functions of the bi-brackets by
\begin{equation*}
    \begin{vmatrix}X_1,\dots,X_l\\Y_1,\dots,Y_l\end{vmatrix}\coloneqq\sum_{\substack{s_1,\dots,s_l>0\\r_1,\dots,r_l\geq0}}\begin{bmatrix}s_1,\dots,s_l\\r_1,\dots,r_l\end{bmatrix}X_1^{s_1-1}\cdots X_l^{s_l-1}Y_1^{r_1}\cdots Y_l^{r_l}.
\end{equation*}
\begin{proposition}
\label{bibracketalternatief}
For an integer $n\in\N$ we define the functions
\begin{equation*}
    E_n(X)\coloneqq e^{nX},\qquad\qquad L_n(X)\coloneqq \frac{e^Xq^n}{1-e^Xq^n}=\sum_{k>0}(e^Xq^n)^k.
\end{equation*}
We have the following two identities between generating functions:
\begin{alignat*}{3}
    i)\qquad&\begin{vmatrix}X_1,\dots,X_l\\Y_1,\dots,Y_l\end{vmatrix}&&=\sum_{u_1>\dots>u_l>0}\prod_{j=1}^lE_{u_j}(Y_j)L_{u_j}(X_j),\\
    ii)\qquad&\begin{vmatrix}X_1,\dots,X_l\\Y_1,\dots,Y_l\end{vmatrix}&&=\sum_{u_1>\dots>u_l>0}\prod_{j=1}^lE_{u_j}(X_{l+1-j}-X_{l+2-j})L_{u_j}(Y_1+\dots+Y_{l-j+1}).
\intertext{The second identity gives us the immediate relation}
    iii)\qquad&\begin{vmatrix}X_1,\dots,X_l\\Y_1,\dots,Y_l\end{vmatrix}&&=\begin{vmatrix}Y_1+\dots+Y_l,\dots,Y_1+Y_2,Y_1\\X_l,X_{l-1}-X_l,\dots,X_1-X_2\end{vmatrix}.
\end{alignat*}
\end{proposition}
\begin{proof}
We have
\begin{align*}
    \begin{vmatrix}X_1,\dots,X_l\\Y_1,\dots,Y_l\end{vmatrix}&=\sum_{\substack{u_1>\dots>u_l>0\\v_1,\dots,v_l>0}}\sum_{\substack{s_1,\dots,s_l>0\\r_1,\dots,r_l\geq 0}}\prod_{j=1}^l\frac{u_j^{r_j}}{r_j!}\frac{v_j^{s_j-1}}{(s_j-1)!}q^{u_jv_j}X_j^{s_j-1}Y_j^{r_j}\\
    &=\sum_{\substack{u_1>\dots>u_l>0\\v_1,\dots,v_l>0}}\prod_{j=1}^l e^{u_jY_j}e^{v_jX_j}q^{u_jv_j}\\
    &=\sum_{\substack{u_1>\dots>u_l>0}}\left(\sum_{v_1,\dots,v_l>0}\prod_{j=1}^l e^{u_jY_j}\left(e^{X_j}q^{u_j}\right)^{v_j}\right)\\
    &=\sum_{u_1>\dots>u_l>0}\prod_{j=1}^l e^{u_jY_j}\frac{e^{X_j}q^{u_j}}{1-e^{X_j}q^{u_j}}.
\end{align*}
This proves the first identity. For the second identity, we apply the transformations $u_j=u_1'+\dots+u_{l-j+1}'$ and $v_j=v_{l-j+1}'-v_{l-j+2}'$ to
\begin{align*}
    \begin{vmatrix}X_1,\dots,X_l\\Y_1,\dots,Y_l\end{vmatrix}&=\sum_{\substack{u_1>\dots>u_l>0\\v_1,\dots,v_l>0}}\prod_{j=1}^l e^{u_jY_j}e^{v_jX_j}q^{u_jv_j}.
\end{align*}
Under these transformations we have $v'_j=v_1+\dots+v_{l-j+1}$, so that
\begin{align*}
    q^{u_1v_1+\dots+u_lv_l}&=q^{(u_1'+\dots+u_l')v_1+(u_1'+\dots+u_{l-1}')v_1+\dots+(u_1'+u_2')v_{l-1}+u_1'v_l}\\
    &=q^{u_1'(v_1+\dots+v_l)+u_2'(v_1+\dots+v_{l-1})+\dots+u_{l-1}'(v_1+v_2)+u_l'v_l}\\
    &=q^{u_1'v_1'+u_2'v_2'+\dots+u_l'v_l'}.
\end{align*}
Additionally, the conditions $u_1>\dots>u_l>0$ and $v_1,\dots,v_l>0$ become $u_1',\dots,u_l'>0$ and $v_1'>\dots>v_l'>0$. Combining the parts gives us
\begin{align*}
    \begin{vmatrix}X_1,\dots,X_l\\Y_1,\dots,Y_l\end{vmatrix}&=\sum_{\substack{v_1'>\dots>v_l'>0\\u_1',\dots,u_l'>0}}\prod_{j=1}^le^{(v_{l-j+1}'-v_{l-j+2}')X_j}e^{(u_1'+\dots+u_{l-j+1}')Y_j}q^{u_j'v_j'}\\
    &=\sum_{\substack{v_1'>\dots>v_l'>0\\u_1',\dots,u_l'>0}}\prod_{j=1}^le^{v_j'(X_{l-j+1}-X_{l-j+2})}e^{u_j'(Y_1+\dots+Y_{l-j+1})}q^{u_j'v_j'}\\
    &=\sum_{\substack{v_1'>\dots>v_l'>0\\u_1',\dots,u_l'>0}}\prod_{j=1}^lE_{v'_j}(X_{l-j+1}-X_{l-j+2})L_{u'_j}(Y_1+\dots+Y_{l-j+1})\\
    &=\begin{vmatrix}Y_1+\dots+Y_l,\dots,Y_1+Y_2,Y_1\\X_l,X_{l-1}-X_l,\dots,X_1-X_2\end{vmatrix}.\qedhere
\end{align*}
\end{proof}

\begin{corollary}
\label{generatingpartitionexample}
For length $l=1$ and $l=2$, respectively, we have
\begin{equation*}
    \begin{vmatrix}X\\Y\end{vmatrix}=\begin{vmatrix}Y\\X\end{vmatrix},\qquad\qquad \begin{vmatrix}X_1,X_2\\Y_1,Y_2\end{vmatrix}=\begin{vmatrix}Y_1+Y_2,Y_1\\X_2,X_1-X_2\end{vmatrix}.
\end{equation*}
\end{corollary}
This gives us the following relations between bi-brackets.
\begin{corollary}
\label{length1partition}
For $s,s_1,s_2>0$ and $r,r_1,r_2\geq 0$ we have 
\begin{align*}
    \begin{bmatrix}s\\r\end{bmatrix}&=\begin{bmatrix}r+1\\s-1\end{bmatrix}\\
\intertext{and}
    \begin{bmatrix}s_1,s_2\\r_1,r_2\end{bmatrix}&=\sum_{\substack{0\leq j\leq r_1 \\ 0\leq k \leq s_2-1}}(-1)^k{s_1-1+k\choose k}{r_2+j\choose j}\bbracket{r_2+1+j,r_1+1-j\\s_2-1-k,s_1-1+k}.
\end{align*}
\end{corollary}
\begin{proof}
For length $l=1$, the statement follows by considering the coefficient of $X^{s-1}Y^r$ in the equality $\begin{vmatrix}X\\Y\end{vmatrix}=\begin{vmatrix}Y\\X\end{vmatrix}$. For length $l=2$, the second identity in Corollary \ref{generatingpartitionexample} states that
\begin{equation*}
    \sum_{\substack{s_1,s_2>0\\r_1,r_2\geq 0}}\begin{bmatrix}s_1,s_2\\r_1,r_2\end{bmatrix}X_1^{s_1-1}X_2^{s_2-1}Y_1^{r_1}Y_2^{r_2}=\sum_{\substack{s_1,s_2>0\\r_1,r_2\geq 0}}\begin{bmatrix}s_1,s_2\\r_1,r_2\end{bmatrix}(Y_1+Y_2)^{s_1-1}Y_1^{s_2-1}X_2^{r_1}(X_1-X_2)^{r_2}.
\end{equation*}
When we expand the left-hand side we get
\begin{equation*}
    (Y_1+Y_2)^{s_1-1}Y_1^{s_2-1}X_2^{r_1}(X_1-X_2)^{r_2}=\sum_{k=0}^{r_2}\sum_{j=0}^{s_1-1}(-1)^{k}{s_1-1\choose j}{r_2\choose k}Y_1^{s_2-1+j}Y_2^{s_1-1-j}X_1^{r_2-k}X_2^{r_1+k}.
\end{equation*}
The statement then follows by considering the coefficient of $X_1^{s_1-1}X_2^{s_2-1}Y_1^{r_1}Y_2^{r_2}$. 
\end{proof}
The relation between bi-brackets arising from Proposition \ref{bibracketalternatief}.$iii$ is called the \textit{partition relation} of bi-brackets. For lengths $l=1,2$ this relation is given in the corollary above.\\

\noindent\textbf{Example.} We have the following partition relations between bi-brackets:
\begin{align*}
    \bbracket{3,1\\1,0}&=\bbracket{1,2\\0,2}+\bbracket{2,1\\0,2},\\
    \bbracket{2,2\\1,1}&=\bbracket{2,2\\1,1}-2\bbracket{2,2\\0,2}+2\bbracket{3,1\\1,1}-4\bbracket{3,1\\0,2}.
\end{align*}

\noindent\textbf{Example.} Suppose that $s_1=s_2=1$. When we apply the partition relation to $\bbracket{1,1\\r_1,r_2}$, all bi-brackets in the resulting linear combination are of the form $\bbracket{r_2+1+j,r_1+1-j\\0\phantom{++.},\phantom{++.}0}=\bbracket{r_2+1+j,r_1+1-j}$. This means that all bi-brackets of the form $\bbracket{1,1\\r_1,r_2}$ are a linear combination of mono-brackets.\\

The operator $q\frac{d}{dq}$ from Section \ref{subsection:derivations} also acts on the bi-brackets.
\begin{proposition}
\label{dqonbibrackets}
We have
\begin{equation*}
    d_q\bbracket{s_1,\dots,s_l\\r_1,\dots,r_l}=\sum_{k=1}^ls_k(r_k+1)\bbracket{s_1,\dots,s_{k-1},s_k+1,s_{k+1},\dots,s_l\\r_1,\dots,r_{k-1},r_k+1,r_{k+1},\dots,r_l}.
\end{equation*}
\end{proposition}
\begin{proof}
This follows directly from the definition and the fact that $d_q q^k=kq^k$.
\end{proof}

\subsection{The stuffle and shuffle product}
The bi-brackets satisfy a stuffle relation which is similar to that of the mono-brackets. We will first consider the stuffle product for two length $1$ bi-brackets.
\begin{lemma}
\label{productgeneratingbibracket}
The product of two generating functions of length $l=1$ is given by
\begin{align*}
    \begin{vmatrix}X_1\\Y_1\end{vmatrix}\cdot\begin{vmatrix}X_2\\Y_2\end{vmatrix}&=\begin{vmatrix}X_1,X_2\\Y_1,Y_2\end{vmatrix}+\begin{vmatrix}X_2,X_1\\Y_2,Y_1\end{vmatrix}+\sum_{k=1}^\infty \frac{B_k}{k!}(X_1-X_2)^{k-1}\begin{vmatrix}X_1+X_2\\Y_1\end{vmatrix}+\sum_{k=1}^\infty \frac{B_k}{k!}(X_2-X_1)^{k-1}\begin{vmatrix}X_1+X_2\\Y_2\end{vmatrix}\\
    &\quad+\frac{1}{X_1-X_2}\left(\begin{vmatrix}X_1\\Y_1+Y_2\end{vmatrix}-\begin{vmatrix}X_2\\Y_1+Y_2\end{vmatrix}\right).
\end{align*}
\end{lemma}
\begin{proof}
According to Proposition \ref{bibracketalternatief} we have
\begin{align*}
    \begin{vmatrix}X_1\\Y_1\end{vmatrix}\cdot\begin{vmatrix}X_2\\Y_2\end{vmatrix}&=\sum_{u_1,u_2>0}E_{u_1}(Y_1)E_{u_2}(Y_2)L_{u_1}(X_1)L_{u_2}(X_2)\\
    &=\sum_{u_1>u_2>0}\cdots\hphantom{i}+\sum_{u_2>u_1>0}\cdots\hphantom{i}+\sum_{u_1=u_2>0}\cdots\\
    &=\begin{vmatrix}X_1,X_2\\Y_1,Y_2\end{vmatrix}+\begin{vmatrix}X_2,X_1\\Y_2,Y_1\end{vmatrix}+\sum_{k>0}E_{k}(Y_1)E_{k}(Y_2)L_{k}(X_1)L_{k}(X_2).
\end{align*}
We directly see that $E_{k}(Y_1)E_{k}(Y_2)=E_k(Y_1+Y_2)$. To evaluate $L_{k}(X_1)L_{k}(X_2)$, note that this expression is equivalent to identity (\ref{eq:emachtbernoulli}) in the proof of Lemma \ref{homomorphismmonobracketsl=1}. This gives us
\begin{align*}
    \sum_{k>0}E_{k}(Y_1)E_{k}(Y_2)L_{k}(X_1)L_{k}(X_2)&=\sum_{k>0}E_k(Y_1+Y_2)\left(\sum_{n>0}\frac{B_n}{n!}(X_1-X_2)^{n-1}L_k(X_1)\right.
    \\
    &\quad+\left.\sum_{n>0}\frac{B_n}{n!}(X_2-X_1)^{n-1}L_k(X_2)+\frac{L_k(X_1)-L_k(X_2)}{X_1-X_2}\right)\\
    &=\sum_{k=1}^\infty \frac{B_k}{k!}(X_1-X_2)^{k-1}\begin{vmatrix}X_1+X_2\\Y_1\end{vmatrix}\\
    &\quad+\sum_{k=1}^\infty \frac{B_k}{k!}(X_2-X_1)^{k-1}\begin{vmatrix}X_1+X_2\\Y_2\end{vmatrix}\\
    &\quad+\frac{1}{X_1-X_2}\left(\begin{vmatrix}X_1\\Y_1+Y_2\end{vmatrix}-\begin{vmatrix}X_2\\Y_1+Y_2\end{vmatrix}\right),
\end{align*}
which completes the proof.
\end{proof}
Lemma \ref{productgeneratingbibracket} implies the following relation between bi-brackets. 
\begin{proposition}
\label{stufflel=1bibracket}
For $s_1,s_2>0$ and $r_1,r_2\geq0$ we have
    \begin{align*}
    \bbracket{s_1\\r_1}\cdot\bbracket{s_2\\r_2}&=\bbracket{s_1,s_2\\r_1,r_2}+\bbracket{s_2,s_1\\r_2,r_1}+{r_1+r_2\choose r_1}\bbracket{s_1+s_2\\r_1+r_2}\\
    &\quad+{r_1+r_2\choose r_1}\sum_{j=1}^{s_1}\lambda_{s_1,s_2}^j\bbracket{j\\r_1+r_2}+{r_1+r_2\choose r_1}\sum_{j=1}^{s_2}\lambda_{s_2,s_1}^j\bbracket{j\\r_1+r_2}
    \end{align*}
where the numbers $\lambda_{a,b}^j$ from Section \ref{subsection:algebraofmonobrackets} are defined as
\begin{equation*}
    \lambda_{a,b}^j=(-1)^{b-1}{a+b-j-1 \choose a-j}\frac{B_{a+b-j}}{(a+b-j)!}.
\end{equation*}
\end{proposition}
\begin{proof}
We once again look at the coefficient of $X_1^{s_1-1}X_2^{s_2-1}Y_1^{r_1}Y_2^{r_2}$ on both sides of the identity from Lemma \ref{productgeneratingbibracket}. The proof is analogous to that of Lemma \ref{homomorphismmonobracketsl=1}.
\end{proof}
We will refer to the product in Proposition \ref{stufflel=1bibracket} as the \textit{stuffle product} (between two length $1$ bi-brackets). We can combine this stuffle product with the partition relation to construct another way of calculating the product between bi-brackets. We will refer to this product as the shuffle product between bi-brackets (of length $1$).\\

For any bi-bracket $\bbracket{s_1,\dots,s_l\\r_1,\dots,r_l}$, let $P\!\bbracket{s_1,\dots,s_l\\r_1,\dots,r_l}$ denote linear combination of brackets corresponding to the partition relation. For lengths $l=1,2$, the linear combinations $P\!\bbracket{s\\r}$ and $P\!\bbracket{s_1,s_2\\r_1,r_2}$ are given in Corollary \ref{length1partition}. We define the \textit{shuffle product} to be the identity
\begin{align*}
    \bbracket{s_1\\r_1}\cdot\bbracket{s_2\\r_2}&=P\left(P\!\bbracket{s_1\\r_1}\cdot P\!\bbracket{s_2\\r_2}\right)\\
    &=\sum_{\substack{0\leq j\leq s_1-1\\0\leq k\leq r_2}}(-1)^k{r_1+k\choose k}{s_2-1+j\choose j}\bbracket{s_2+j,s_1-j\\r_2-k,r_1+k}\\
    &\quad+\sum_{\substack{0\leq j\leq s_2-1\\0\leq k\leq r_1}}(-1)^k{r_2+k\choose k}{s_1-1+j\choose j}\bbracket{s_1+j,s_2-j\\r_1-k,r_2+k}\\
    &\quad+{s_1+s_2-2\choose s_1-1}\bbracket{s_1+s_2-1\\r_1+r_2+1}+{s_1+s_2-2\choose s_1-1}\sum_{j=0}^{r_1}\lambda_{r_1,r_2+1}^j\bbracket{s_1+s_2-1\\j}\\
    &\quad+{s_1+s_2-2\choose s_1-1}\sum_{j=0}^{r_2}\lambda_{r_2,r_1+1}^j\bbracket{s_1+s_2-1\\j}.
\end{align*}
In the next section we will generalise the stuffle and shuffle product to higher length.
\subsection{The algebra of bi-brackets}
Consider the alphabet $A_\BD\coloneqq\{z_{s,r}\mid s>0, r\geq 0\}$. For any two letters $z_{s_1,r_1}$ and $z_{s_2,r_2}$, define the product
\begin{equation*}
    z_{s_1,r_1}\diamond z_{s_2,r_2}\coloneqq {r_1+r_2\choose r_1}z_{s_1+s_2,r_1+r_2}+{r_1+r_2\choose r_1}\sum_{j=1}^{s_1}\lambda_{s_1,s_2}^jz_{j,r_1+r_2}+{r_1+r_2\choose r_1}\sum_{j=1}^{s_2}\lambda_{s_2,s_1}^jz_{j,r_1+r_2}.
\end{equation*}
On the space $\Q\langle A_\BD\rangle$ we define the stuffle product as
\begin{equation*}
    z_{s_1,r_1}v\ast z_{s_2,r_2}w= z_{s_1,r_1}(v\ast z_{s_2,r_2}w) +z_{s_2,r_2}(z_{s_1,r_1}v\ast w)+(z_{s_1,r_1}\diamond z_{s_2,r_2})(v\ast w).
\end{equation*}
This defines the quasi-shuffle algebra $\Ha^\BD_\ast$.
\begin{proposition}
The map $\begin{bmatrix}\ \cdot\ \end{bmatrix}\colon z_{s_1,r_1}\cdots z_{s_l,r_l}\mapsto \begin{bmatrix}s_1,\dots,s_l\\r_1,\dots,r_l\end{bmatrix}$ is a homomorphism of $\Q$-algebras. This means that for any $v,w\in \Ha_\ast^\BD$ we have
\begin{equation*}
    \begin{bmatrix}v\ast w\end{bmatrix}=\begin{bmatrix}v\end{bmatrix}\cdot \begin{bmatrix}w\end{bmatrix}.
\end{equation*}
\end{proposition}
\begin{proof}
The proof is analogue to the proof of Proposition \ref{homomorphismmonobrackets} and will be omitted.
\end{proof}

Since the bi-brackets are also equipped with a partition relation (which gives rise to the shuffle product), it makes sense to also define this relation on $\Q\langle A_\BD\rangle$. Let $P\colon \Q\langle A_\BD\rangle\to \Q\langle A_\BD\rangle$ be the map which sends a word to the linear combination of words corresponding to the partition relation. For example, in the cases $l=1,2$, the map $P$ is given by
\begin{align*}
    P(z_{s_1,r_1})&=z_{r_1+1,s_1-1},\\
    P(z_{s_1,r_1}z_{s_2,r_2})&=\sum_{\substack{0\leq l\leq r_1\\0\leq k\leq s_2-1}}(-1)^k{s_1-1+k\choose k}{r_2+l\choose l} z_{r_2+1+l,s_2-1-k}z_{r_1+1-l,s_1-1+k}.
\end{align*}
Then by construction, the map $\begin{bmatrix}\ \cdot\ \end{bmatrix}\colon z_{s_1,r_1}\cdots z_{s_l,r_l}\mapsto \begin{bmatrix}s_1,\dots,s_l\\r_1,\dots,r_l\end{bmatrix}$ is invariant under $P$. This means that the product $\shuffle$, defined on two words $v,w$ by
\begin{equation*}
    v\shuffle w= P\big(P(v)\ast P(w)\big)
\end{equation*}
and extended linearly to $\Q\langle A_\BD\rangle$, will also satisfy $\begin{bmatrix}v\shuffle w\end{bmatrix}=\begin{bmatrix}v\end{bmatrix}\cdot\begin{bmatrix}w\end{bmatrix}$. For example, the stuffle and shuffle product between length $1$ bi-brackets are given by
\begin{align*}
    z_{s_1,r_1}\ast z_{s_2,r_2}&=z_{s_1,r_1}z_{s_2,r_2}+z_{s_2,r_2}z_{s_1,r_1}+{r_1+r_2\choose r_1}z_{s_1+s_2,r_1+r_1}\\
    &\quad+{r_1+r_2\choose r_1}\sum_{j=1}^{s_1}\lambda_{s_1,s_2}^jz_{j,r_1+r_2}+{r_1+r_2\choose r_1}\sum_{j=1}^{s_1}\lambda_{s_2,s_1}^jz_{j,r_1+r_2}\\
\intertext{and}
    z_{s_1,r_1}\shuffle z_{s_2,r_2}&=\sum_{\substack{1\leq l\leq s_1\\ 0\leq k\leq r_2}}{s_1+s_2-1-l \choose s_2-1}{r_1+r_2-k \choose r_1}(-1)^{r_2-k}z_{s_1+s_2-l,k}z_{l,r_1+r_2-k}\\
    &\quad+\sum_{\substack{1\leq l\leq s_2\\ 0\leq k\leq r_1}}{s_1+s_2-1-l \choose s_1-1}{r_1+r_2-k \choose r_2}(-1)^{r_1-k}z_{s_1+s_2-l,k}z_{l,r_1+r_2-k}\\
    &\quad+{s_1+s_2-2\choose s_1-1}z_{s_1+s_2-1,r_1+r_2+1}\\
    &\quad+{s_1+s_2-2\choose s_1-1}\sum_{j=0}^{r_1}\lambda_{r_1,r_2+1}^jz_{s_1+s_2-1,j}+{s_1+s_2-2\choose s_1-1}\sum_{j=0}^{r_2}\lambda_{r_2,r_1+1}^jz_{s_1+s_2-1,j}.
\end{align*}
\subsection{Connections with multiple zeta values}
We can rewrite the bi-brackets, similar to the mono-brackets, using the polynomials $P_s(x)$. We obtain the alternate form
\begin{equation*}
    \bbracket{s_1,\dots,s_l\\r_1,\dots,r_l}=\frac{1}{r_1!(s_1-1)!\dots r_l!(s_l-1)!}\sum_{n_1>\dots n_l>0}\frac{n_1^{r_1}P_{s_1}(q^{n_1})\cdots n_l^{r_l}P_{s_l}(q^{n_l})}{(1-q^{n_1})^{s_1}\cdots (1-q^{n_l})^{s_l}}.
\end{equation*}
This means that we have
\begin{equation*}
    \lim_{q\to 1^-}(1-q)^{s_1+\dots+s_l}\bbracket{s_1,\dots,s_l\\r_1,\dots,r_l}=\frac{1}{r_1!\cdots r_l!}\zeta(s_1-r_1,\dots,s_l-r_l),
\end{equation*}
whenever $s_1-r_1>1$ and $s_j-r_j\geq 1$. And so the bi-brackets are too a generalisation of the multiple zeta values.

Consider the restricted alphabet $A'\subset A_\BD$ consisting of the letters $z_{s,0}$. A word $z_{s}v\in \Q\langle A'\rangle$ is called admissible if $s>1$.
\begin{proposition}
\label{prop:monoshufismzv}
Let $v,w\in\Q\langle A'\rangle$ be admissible. Then we have the identity
\begin{equation*}
    \lim_{q\to1^-}(1-q)^{{|v| +|w|}}\bbracket{v\ast w}=\zeta(v\ast w),
\end{equation*}
where $v\ast w$ on the left-hand side denotes the stuffle product between bi-brackets, and between multiple zeta values on the right-hand side.
\end{proposition}
\begin{proof}
We observe that the stuffle product $v\ast w$ for words $v,w\in\Q\langle A'\rangle$ between bi-brackets is the same as the stuffle product for the mono-brackets $\overline{v}\ast \overline{w}$, with
\begin{equation*}
    \overline{v}=\overline{\vphantom{t}z_{s_1,0}\cdots z_{s_1,0}}=z_{s_1}\cdots z_{s_2}\in\Ha_\ast^\MD.
\end{equation*}
The statement of the proposition then follows directly from Proposition \ref{mbtomzv}.
\end{proof}
A similar statement is true regarding the shuffle product on the algebra $\Q\langle A'\rangle$. The proof can be found in \cite[Thm. 3, p. 9]{qbrackets}.
\begin{proposition}
\label{prop:bibshufismzv}
For two admissible words $v,w\in \Q\langle A'\rangle$ we have
\begin{equation*}
    \lim_{q\to 1^-}(1-q)^{|v|+|w|}\bbracket{v \shuffle w}=\zeta (v\shuffle w).
\end{equation*}
\end{proposition}
This means that both the stuffle and shuffle products between bi-brackets are generalisations of the stuffle and shuffle between multiple zeta values.

\section{Principal result}
In Section \ref{subsection:algebraofmonobrackets} we showed that the mono-brackets possess a stuffle relation, which generalises the stuffle relation between multiple zeta values. Proposition \ref{prop:bibshufismzv} then shows that the mono-brackets also potentially satisfy a shuffle relation, when we consider the mono-brackets as a special type of bi-brackets. However, the brackets appearing in the product $v\shuffle w$ are not (by construction) all mono-brackets. For example, the shuffle product of two length $1$ brackets is given by
\begin{align*}
    \bbracket{z_{s,0}\shuffle z_{t,0}}=\bbracket{s}\cdot\bbracket{t}&=\sum_{\substack{1\leq l\leq s}}{s+t-1-l \choose t-1}\bbracket{s+t-l,l}    +\sum_{\substack{1\leq l\leq t}}{s+t-1-l \choose s-1}\bbracket{s+t-l,l}\\
    &\quad+{s+t-2\choose s-1}\bbracket{s+t-1\\1}-{s+t-2\choose s-1}\bbracket{s+t-1}.
\end{align*}
This means that the shuffle product fails to be a relation solely between mono-brackets.

However, we can apply Proposition \ref{dqonbibrackets} and Theorem \ref{invarianceofd} to see that
\begin{equation*}
    \bbracket{s+t-1\\1}=\frac{1}{s+t-2}\;d_q\negphantom{.}\bbracket{s+t-2}
\end{equation*}
is contained in $\MD$. Generally, we could wonder if these ``failure''-brackets can always be written purely in terms of mono-brackets. In this light, Bachmann \cite[Conj. 4.3, p. 11]{bachbi} conjectured the following.
\begin{conjecture*}
\label{conjecture:bdismd}
Every bi-bracket can be written as a $\Q$-linear combination of mono-brackets, or equivalently, the spaces $\MD$ and $\BD$ coincide. I.e., we have $BD=MD$.
\end{conjecture*}
Though not a lot is known, there are some cases which have been  solved. We will show that all bi-brackets of length $1$ can be written as linear combinations of mono-brackets.
\begin{theorem}
\label{bdisinmdl=1}
For $s>0$ and $r\geq 0$ we have $\begin{bmatrix}s\\r\end{bmatrix}\in \MD$.
\end{theorem}
\noindent\textit{Proof.} Since we have $\begin{bmatrix}s\\r\end{bmatrix}=\begin{bmatrix}r+1\\s-1\end{bmatrix}$, we may assume that $s>r$. Then, after applying Proposition \ref{dqonbibrackets} and Theorem \ref{invarianceofd} sufficiently many times, we get
\begin{equation*}
    \pushQED{\qed}\begin{bmatrix}s\\r\end{bmatrix}=\frac{1}{(s-1)r}d_q\bbracket{s-1\\r-1}= \hphantom{l}\cdots\hphantom{l}=C\cdot (d_q)^r\begin{bmatrix}s-r\\0\end{bmatrix}=C\cdot(d_q)^r[s-r]\in \MD.
\qedhere\popQED
\end{equation*}
Now consider the case $l=2$. Bachmann \cite[Prop. 5.9, p. 17]{bachbi} showed that bi-brackets of lower weight $1$ can be written as a $\Q$-linear relation of mono-brackets.
\begin{theorem}
For $s_1,s_2>0$ we have $\bbracket{s_1,s_2\\1,0},\bbracket{s_1,s_2\\0,1}\in\MD$.
\end{theorem}
Additionally, the statement of the conjecture is true for brackets with odd total weight. This statement is the principal result of this thesis.
\begin{theorem}
\label{oddweightl=2}
Let $s_1,s_2>0$ and $r_1,r_2\geq0$ be integers. If $s_1+s_2+r_1+r_2$ is odd, we have $\bbracket{s_1,s_2\\r_1,r_2}\in \MD$.
\end{theorem}

The proof of Theorem \ref{oddweightl=2} is based on Theorem \ref{bdisinmdl=1} and the observation that the bi-brackets occurring in the shuffle and stuffle product $\bbracket{s_1\\r_1}\bbracket{s_2\\r_2}$ all have the same upper weight $s_1+s_2$ and lower weight $r_1+r_2$. This means that it is natural to consider all products $\bbracket{s_1\\r_1}\bbracket{s_2\\r_2}$ for fixed $s_1+s_2=S$ and $r_1+r_2=R$. We will demonstrate the proof of Theorem \ref{oddweightl=2} in an example.\\

\noindent\textbf{Example.} Let us fix $S=3$ and $R=2$. Then all possible stuffle and shuffle products $\bbracket{s_1\\r_1}\bbracket{s_2\\r_2}$ with $s_1+s_2=3$ and $r_1+r_2=2$ are 
\begin{alignat*}{7}
    \bbracket{2\\2}\bbracket{1\\0}&= \bbracket{2,1\\2,0}+&&&&&&&&&&\hphantom{2}\bbracket{1,2\\0,2}&&\qquad+\bbracket{3\\2}-\frac{1}{2}\bbracket{2\\2},\\
    \bbracket{2\\1}\bbracket{1\\1}&=&& \hphantom{2}\bbracket{2,1\\1,1}+&&&&&&\bbracket{1,2\\1,1}&&&&\qquad+2\bbracket{3\\2}-\bbracket{2\\2},\\
    \bbracket{2\\0}\bbracket{1\\2}&=&&&& \hphantom{4}\bbracket{2,1\\0,2}+&&\bbracket{1,2\\2,0}&&&&&&\qquad+\bbracket{3\\2}-\frac{1}{2}\bbracket{2\\2},\\
\intertext{and}
    \bbracket{2\\2}\bbracket{1\\0}&= \bbracket{2,1\\2,0}-&&\hphantom{2}\bbracket{2,1\\1,1}+&&2\bbracket{2,1\\0,2}+&&&&&&\hphantom{2}\bbracket{1,2\\0,2}&&\qquad+\bbracket{2\\3}+\frac{1}{12}\bbracket{2\\1}-\frac{1}{2}\bbracket{2\\2},\\
    \bbracket{2\\1}\bbracket{1\\1}&=&&2\bbracket{2,1\\1,1}-&&4\bbracket{2,1\\0,2}+&&&&\bbracket{1,2\\1,1}-&&2\bbracket{1,2\\0,2}&&\qquad+\bbracket{2\\3}-\frac{1}{6}\bbracket{2\\1},\\
    \bbracket{2\\0}\bbracket{1\\2}&= \bbracket{2,1\\2,0}-&&\hphantom{2}\bbracket{2,1\\1,1}+&&2\bbracket{2,1\\0,2}+&&\bbracket{1,2\\2,0}-&&\bbracket{1,2\\1,1}+&&\hphantom{2}\bbracket{1,2\\0,2}&&\qquad+\bbracket{2\\3}+\frac{1}{12}\bbracket{2\\1}-\frac{1}{2}\bbracket{2\\2}.
\end{alignat*}
Now suppose that the bi-brackets of length $2$ in these relations are unknowns, and all the other terms are known. Solving for the bi-brackets means finding an inverse of the matrix
\begin{equation*}
    M=\begin{pmatrix}1&0&0&0&0&1\\0&1&0&0&1&0\\0&0&1&1&0&0\\1&-1&2&0&0&1\\0&2&-4&0&1&-2\\1&-1&2&1&-1&1\end{pmatrix}.
\end{equation*}
And indeed, $M$ is invertible with its inverse given by
\begin{equation*}
    M^{-1}=\frac{1}{2}\begin{pmatrix}-1& 1& -2& 1& 1& 2\\-2& 4& -4& -2& 0& 4\\-2& 2& -2& 0& 0& 
  2\\2& -2& 4& 0& 0& -2\\2& -2& 4& 2& 0& -4\\3& -1& 
  2& -1& -1& -2\end{pmatrix}.
\end{equation*}
Explicitly, this gives us the following relations between bi-brackets
\begin{align*}
    \bbracket{2,1\\2,0}&=\bbracket{2\\1}\bbracket{1\\1}+\frac{1}{2}\bbracket{3\\2}-2\bbracket{2\\3}+\frac{1}{2}\bbracket{2\\2}-\frac{1}{24}\bbracket{2\\1},\\
    \bbracket{2,1\\1,1}&=-2\bbracket{2\\2}\bbracket{1\\0}+2\bbracket{2\\1}\bbracket{1\\1}-\bbracket{3\\2}-\bbracket{2\\3}+\bbracket{2\\2}-\frac{1}{12}\bbracket{2\\1},\\
    \bbracket{2,1\\0,2}&=-\bbracket{2\\2}\bbracket{1\\0}+\bbracket{2\\1}\bbracket{1\\1}-\bbracket{2\\3}+\frac{1}{2}\bbracket{2\\2}-\frac{1}{12}\bbracket{2\\1},\\
    \bbracket{1,2\\2,0}&=\bbracket{2\\2}\bbracket{1\\0}-\bbracket{2\\1}\bbracket{1\\1}+\bbracket{2\\0}\bbracket{1\\2}-\bbracket{3\\2}+\bbracket{2\\3}+\frac{1}{12}\bbracket{2\\1},\\
    \bbracket{1,2\\1,1}&=2\bbracket{2\\2}\bbracket{1\\0}-\bbracket{2\\1}\bbracket{1\\1}-\bbracket{3\\2}+\bbracket{2\\3}-\frac{1}{12}\bbracket{2\\1},\\
    \bbracket{1,2\\0,2}&=\bbracket{2\\2}\bbracket{1\\0}-\bbracket{2\\1}\bbracket{1\\1}-\frac{3}{2}\bbracket{3\\2}-2\bbracket{2\\3}+\frac{1}{24}\bbracket{2\\1}.
\end{align*}
Theorem \ref{bdisinmdl=1} tells us that all terms on the right-hand side of these equations can be written as $\Q$-linear combinations of mono-brackets. We have shown that the brackets $\bbracket{2,1\\2,0},\bbracket{2,1\\1,1},\dots,\bbracket{1,2\\0,2}$ are contained in $\MD$. \\

We could wonder if we can apply this method for every upper weight $S$ and lower weight $R$. The answer is no. For example, if we let $S=3$ and $R=1$, we consider the relations
\begin{alignat*}{5}
    \bbracket{2\\1}\bbracket{1\\0}&= \bbracket{2,1\\1,0}+&&&&&&\bbracket{1,2\\0,1}&&\qquad+\bbracket{3\\1}-\frac{1}{2}\bbracket{2\\1},\\
    \bbracket{2\\0}\bbracket{1\\1}&=&& \bbracket{2,1\\0,1}+&&\bbracket{1,2\\1,0}&&&&\qquad+\bbracket{3\\1}-\frac{1}{2}\bbracket{2\\1},\\
\intertext{and}
    \bbracket{2\\1}\bbracket{1\\0}&= \bbracket{2,1\\1,0}+&&&&&&\bbracket{1,2\\0,1}&&\qquad+\bbracket{2\\2}-\frac{1}{2}\bbracket{2\\1},\\
    \bbracket{2\\0}\bbracket{1\\1}&=\bbracket{2,1\\1,0}+&&&&\bbracket{1,2\\1,0}-&&\bbracket{1,2\\0,1}&&\qquad+\bbracket{2\\2}-\frac{1}{2}\bbracket{2\\1}.
\end{alignat*}
However, the corresponding matrix $\begin{pmatrix}1&0&0&1\\0&1&1&0\\1&0&0&1\\1&0&1&-1\end{pmatrix}$ is not invertible. It turns out that this system can be solved for the bi-brackets of length $2$ if the fixed total weight $S+R$ is odd.\\

\noindent\textbf{Remark.} We have the following identity between brackets
\begin{equation*}
    \bbracket{2,2\\0,1}=\bbracket{3,2}-\bbracket{2,2,1}+\frac{1}{2}\bbracket{2,2},
\end{equation*}
which arises from the conjecture ``$\MD=\BD$''. When we apply the limit $\lim_{q\to1^-}(1-q)^5$ to this relation, the bi-bracket vanishes since we have
\begin{equation*}
    \lim_{q\to1^-}(1-q)^N\bbracket{s_1,\dots,s_l\\r_1,\dots,r_l}=0
\end{equation*}
whenever $s_1+\dots+s_l<N$. This gives us the (duality) relation $\zeta(3,2)=\zeta(2,2,1)$ between MZVs. This means that this conjecture generates relations between MZVs. It is, unofficially, believed that all possible relations between MZVs arise this way.
\subsection{The proof}
\label{subsection:theproof}
Let us return to the algebraic setting of the bi-brackets. For the proof of Theorem \ref{bdisinmdl=1} we will consider the altered stuffle and shuffle product defined by
\begin{align*}
    z_{s_1,r_1}\astt z_{s_2,r_2}&=z_{s_1,r_1}z_{s_2,r_2}+z_{s_2,r_2}z_{s_1,r_1}\\
\intertext{and}
    z_{s_1,r_1}\shuf z_{s_2,r_2}&=P\big( P(z_{s_1,r_1})\;\astt \;P(z_{s_2,r_2})\big)\\
    &=\sum_{\substack{1\leq l\leq s_1\\ 0\leq k\leq r_2}}{s_1+s_2-1-l \choose s_2-1}{r_1+r_2-k \choose r_1}(-1)^{r_2-k}z_{s_1+s_2-l,k}z_{l,r_1+r_2-k}\\
    &\quad+\sum_{\substack{1\leq l\leq s_2\\ 0\leq k\leq r_1}}{s_1+s_2-1-l \choose s_1-1}{r_1+r_2-k \choose r_2}(-1)^{r_1-k}z_{s_1+s_2-l,k}z_{l,r_1+r_2-k}.
\end{align*}
These are precisely the usual stuffle and shuffle product, where we omit the words of length $1$. These altered stuffle and shuffle products satisfy
\begin{align}
    \bbracket{z_{s_1,r_1}\astt z_{s_2,r_2}}&\equiv\bbracket{z_{s_1,r_1}\ast z_{s_2,r_2}}=\bbracket{z_{s_1,r_1}}\bbracket{ z_{s_2,r_2}}&&\negphantom{mmmmmi.}\equiv 0\qquad\qquad\text{mod } \MD,\\
    \bbracket{z_{s_1,r_1}\shuf z_{s_2,r_2}}&\equiv\bbracket{z_{s_1,r_1}\shuffle z_{s_2,r_2}}=\bbracket{z_{s_1,r_1}}\bbracket{ z_{s_2,r_2}}&&\negphantom{mmmmmi.}\equiv0\qquad\qquad\text{mod } \MD.
\end{align}
This means that, when we are only interested in the bi-brackets modulo $\MD$, it is sufficient to consider these altered products.\\

In the proof of Theorem \ref{oddweightl=2}, the notation ``$\equiv$" means ``modulo the space $\MD$".\\

\noindent\textit{Proof (Theorem \ref{oddweightl=2}).} Let us fix $S,R\in \Z$ with $S\geq2$ and $R\geq 0$, and assume that $S+R$ is odd. For letters $z_{s_1,r_1}$ and $z_{s_2,r_2}$ with $s_1+s_2=S$ and $r_1+r_2=R$, all terms occurring in the products $z_{s_1,r_1} \astt z_{s_2,r_2}$ and $z_{s_1,r_1} \shuf z_{s_2,r_2}$ will have upper weight $S$ and lower weight $R$. Consider the pairs 
\begin{align*}
    Z=\{(z_{s_1,r_1},z_{s_2,r_2}) \mid s_1+s_2=S \text{ and } r_1+r_2=R\}.
\end{align*}
These pairs generate a system of equations, namely
\begin{align}
    0\equiv \bbracket{z_{s_1,r_1}\astt z_{s_2,r_2}}&=\bbracket{z_{s_1,r_1}z_{s_2,r_2}} +\bbracket{z_{s_2,r_2}z_{s_1,r_1}},\label{eq:system11}\\
    0\equiv \bbracket{z_{s_1,r_1}\shuf z_{s_2,r_2}}&=\sum_{\substack{1\leq l\leq s_1\\ 0\leq k\leq r_2}}{s_1+s_2-l-1 \choose s_2-1}{r_1+r_2-k \choose r_1}(-1)^{r_2-k}\bbracket{z_{s_1+s_2-l,k}z_{l,r_1+r_2-k}}\nonumber\\
    &\quad+\sum_{\substack{1\leq l\leq s_2\\ 0\leq k\leq r_1}}{s_1+s_2-l-1 \choose s_1-1}{r_1+r_2-k \choose r_2}(-1)^{r_1-k}\bbracket{z_{s_1+s_2-l,k}z_{l,r_1+r_2-k}}.\label{eq:system13}
\end{align}
This system consists of $2(S-1)(R+1)$ equations, while we also have $|Z|=(S-1)(R+1)$. Hence the matrix corresponding to the system above is non-square. To reduce this system of equations, we apply the relation $$\bbracket{z_{s_1,r_1}z_{s_2,r_2}}\equiv -\bbracket{z_{s_2,r_2}z_{s_1,r_1}},$$ coming from (\ref{eq:system11}), to the terms in the second sum in (\ref{eq:system13}).
Additionally, we use the fact that ${a\choose b}=0$ if $b>a$ or if $b<0$, apply the transformation $k\mapsto R+1-k$, and use the alternative parametrization
\begin{equation*}
    Z=\{(z_{S-n,R+1-m},z_{n,m-1}) \mid 1\leq n\leq S-1 \text{ and } 1\leq m\leq R+1\}.
\end{equation*}
We will show that this new system of equations, which is given by
\begin{align}
\label{system2}
    0&\equiv\sum_{\substack{1\leq l\leq S-1\\ 1\leq k\leq R+1}}{S-1-l \choose n-1}{k-1 \choose R+1-m}(-1)^{R+m+k}\begin{bmatrix}z_{S-l,R+1-k}z_{l,k-1}\end{bmatrix}\nonumber\\
    &\quad-\sum_{\substack{1\leq l\leq S-1\\ 1\leq k\leq R+1}}{S-1-l \choose n-l}{k-1 \choose m-1}(-1)^{m+k}\begin{bmatrix}z_{l,k-1}z_{S-l,R+1-k}\end{bmatrix},
\end{align}
has a solution with respect to the terms $\bbracket{z_{s_1,r_1}z_{s_2,r_2}}$.\\

Define the $N\times N$-matrices $P_N$ and $Q_N$ by 
\begin{equation*}
    (P_N)_{i,j}={N-j\choose i-1},\qquad\qquad (Q_N)_{i,j}=(-1)^{N+i+j+1}{j-1 \choose N-i}.
\end{equation*}
For an $N\times N$-matrix $A$, define $\rho$ and $\sigma$ to be the horizontal and vertical reflections
\begin{equation*}
    (A^\rho)_{i,j}=A_{N+1-i,j},\qquad\qquad(A^\sigma)_{i,j}=A_{i,N+1-j}.
\end{equation*}
Then the matrix corresponding to System (\ref{system2}) is given by $P_{S-1}\otimes Q_{R+1}-P_{S-1}^{\rho\sigma}\otimes Q_{R+1}^{\rho\sigma}$.\\

We claim that $P_{S-1}\otimes Q_{R+1}-P_{S-1}^{\rho\sigma}\otimes Q_{R+1}^{\rho\sigma}$ is invertible, with its inverse given by 
\begin{equation*}
    \frac{1}{2}\left(M_{S,R}+Q_{S-1}\otimes P_{R+1}-Q_{S-1}^{\rho\sigma}\otimes P_{R+1}^{\rho\sigma}\right).
\end{equation*}
Here, $M_{S,R}$ is the matrix defined by $M_{S,R}=J_{S-1}\otimes J_{R+1}$, with $(J_n)_{i,j}=\begin{cases}(-1)^{i+1}&\text{ if }i=j\\0&\text{ if } i\neq j\end{cases}$.

Consider the product
\begin{align*}
    \Sigma\coloneqq\left(P_{S-1}\otimes \right.&\left.Q_{R+1}-P_{S-1}^{\rho\sigma}\otimes\ Q_{R+1}^{\rho\sigma}\right)\left(Q_{S-1}\otimes P_{R+1}-Q_{S-1}^{\rho\sigma}\otimes P_{R+1}^{\rho\sigma}\right)=P_{S-1}Q_{S-1}\otimes Q_{R+1}P_{R+1}\\
    &-P_{S-1}Q_{S-1}^{\rho\sigma}\otimes Q_{R+1}P_{R+1}^{\rho\sigma}-P_{S-1}^{\rho\sigma}Q_{S-1}\otimes Q_{R+1}^{\rho\sigma}P_{R+1}+P_{S-1}^{\rho\sigma}Q_{S-1}^{\rho\sigma}\otimes Q_{R+1}^{\rho\sigma}P_{R+1}^{\rho\sigma}.
\end{align*}
We have the following equalities (see Lemma \ref{binomialidentities}):
\begin{align*}
    (P_{S-1}Q_{S-1})_{i,j}&=\sum_{k=1}^{S-1} (-1)^{k+j+S}{S-1-k \choose i-1}{j-1\choose S-1-k}=\begin{cases}1 & \qquad\text{ if }i=j\\
    0&\qquad\text{ if }i\neq j\end{cases},\\
    (Q_{R+1}P_{R+1})_{i,j}&=\sum_{k=1}^{R+1} (-1)^{k+i+R}{k-1 \choose R+1-i}{R+1-j\choose k-1}=\begin{cases}1 & \qquad\text{ if }i=j\\
    0&\qquad\text{ if }i\neq j\end{cases},\\
    (P_{S-1}Q^{\rho\sigma}_{S-1})_{i,j}&=\sum_{k=1}^{S-1} (-1)^{k+j+S}{S-1-k \choose i-1}{S-1-j\choose k-1}=(-1)^{j+S+1}{j-1\choose S-1-i},\\
    (Q_{R+1}P^{\rho\sigma}_{R+1})_{i,j}&=\sum_{k=1}^{R+1} (-1)^{k+i+R}{j-1\choose R+1-k}{k-1 \choose R+1-i}=(-1)^{i+1}{R+1-j\choose i-1},\\
    (P_{S-1}^{\rho\sigma}Q_{S-1})_{i,j}&=\sum_{k=1}^{S-1}(-1)^{k+j+S}{j-1\choose S-1-k}{k-1\choose S-1-i}=(-1)^{j+1}{S-1-j\choose i-1},\\
    (Q_{R+1}^{\rho\sigma}P_{R+1})_{i,j}&=\sum_{k=1}^{R+1}(-1)^{k+i+R}{R+1-k\choose i-1}{R+1-j\choose k-1}=(-1)^{i+R+1}{j-1\choose R+1-i},\\
    (P^{\rho\sigma}_{S-1}Q^{\rho\sigma}_{S-1})_{i,j}&=\sum_{k=1}^{S-1} (-1)^{k+j+S}{k-1 \choose S-1-i}{S-1-j\choose k-1}=\begin{cases}1 & \qquad\text{ if }i=j\\
    0&\qquad\text{ if }i\neq j\end{cases},\\
    (Q^{\rho\sigma}_{R+1}P^{\rho\sigma}_{R+1})_{i,j}&=\sum_{k=1}^{R+1} (-1)^{k+i+R}{R+1-k \choose i-1}{j-1\choose R+1-k}=\begin{cases}1 & \qquad\text{ if }i=j\\
    0&\qquad\text{ if }i\neq j\end{cases}.
\end{align*}
Note that the last four identities follow from the first four.
On the other hand, we have 
\begin{align*}
    (P_{S-1}J_{S-1})_{i,j}&=(-1)^{j+1}{S-1-j\choose i-1}\\
    (Q_{R+1}J_{R+1})_{i,j}&=(-1)^{i+R+1}{j-1\choose R+1-i}\\
    (P^{\rho\sigma}_{S-1}J_{S-1})_{i,j}&=(-1)^{j+1}{j-1\choose S-1-i}\\
    (Q^{\rho\sigma}_{R+1}J_{R+1})_{i,j}&=(-1)^{i+R+1}{R+1-j\choose i-1}.
\end{align*}
Since we assumed that $S+R$ is odd, either $S$ is even and $R$ is odd, or $S$ is odd and $R$ is even. In the first case we have $P_{S-1}Q^{\rho\sigma}_{S-1}=P^{\rho\sigma}_{S-1}J_{S-1}$ and $Q_{R+1}P^{\rho\sigma}_{R+1}=-Q^{\rho\sigma}_{R+1}J_{R+1}$. In the second case we get $P_{S-1}Q^{\rho\sigma}_{S-1}=-P^{\rho\sigma}_{S-1}J_{S-1}$ and $Q_{R+1}P^{\rho\sigma}_{R+1}=Q^{\rho\sigma}_{R+1}J_{R+1}$.
In either case, we have the following equality
\begin{equation*}
    P_{S-1}Q^{\rho\sigma}_{S-1}\otimes Q_{R+1}P^{\rho\sigma}_{R+1}=-P^{\rho\sigma}_{S-1}J_{S-1}\otimes Q^{\rho\sigma}_{R+1}J_{R+1}.
\end{equation*}
Additionally, we always have
\begin{equation*}
    P_{S-1}^{\rho\sigma}Q_{S-1}\otimes Q_{R+1}^{\rho\sigma}P_{R+1}=P_{S-1}J_{S-1}\otimes Q_{R+1}J_{R+1}.
\end{equation*}

Combining all the parts, we obtain
\begin{align*}
    \Sigma&=2I_{(S-1)(R+1)}-P_{S-1}Q_{S-1}^{\rho\sigma}\otimes Q_{R+1}P_{R+1}^{\rho\sigma}-P_{S-1}^{\rho\sigma}Q_{S-1}\otimes Q_{R+1}^{\rho\sigma}P_{R+1}\\
    &=2I_{(S-1)(R+1)}+P^{\rho\sigma}_{S-1}J_{S-1}\otimes Q^{\rho\sigma}_{R+1}J_{R+1}-P_{S-1}J_{S-1}\otimes Q_{R+1}J_{R+1}\\
    &=2I_{(S-1)(R+1)}-\left(P_{S-1}\otimes Q_{R+1}-P^{\rho\sigma}_{S-1}\otimes Q^{\rho\sigma}_{R+1}\right)M_{S,R}.
\end{align*}
This proves the desired claim:
\begin{equation*}
    \frac{1}{2}\left(P_{S-1}\otimes Q_{R+1}-P_{S-1}^{\rho\sigma}\otimes Q_{R+1}^{\rho\sigma}\right)\left(M_{S,R}+Q_{S-1}\otimes P_{R+1}-Q_{S-1}^{\rho\sigma}\otimes P_{R+1}^{\rho\sigma}\right)=I_{(S-1)(R+1)}.
\end{equation*}
To prove the statement of the theorem, we note that an inverse ensures us that the following equation has an unique solution:
\begin{equation*}
    (P_{S-1}\otimes Q_{R+1}-P_{S-1}^{\rho\sigma}\otimes Q_{R+1}^{\rho\sigma})\mathbf{x}=\mathbf{0},
\end{equation*}
where $\mathbf{x}$ is the vector consisting of the words of length $2$ corresponding to $Z$. But this precisely means that 
\begin{equation*}
    \bbracket{z_{s_1,r_1}z_{s_2,r_2}}\equiv 0\qquad\qquad\text{ mod } \MD
\end{equation*}
for all $s_1,s_2>1$ and $r_1,r_2\geq 0$ with $s_1+s_2=S$ and $r_1+r_2=R$.
This completes the proof.\qed\\

We are left to prove the following lemma.
\begin{lemma}
\label{binomialidentities}
For integers $n\geq1$ and $1\leq i,j\leq n$ we have
\begin{align}
    \sum_{k=1}^n(-1)^{k+j+n+1} {n-k \choose i-1}{j-1\choose n-k}&=\begin{cases}1 & \qquad\text{ if }i=j\label{binomials1}\\
    0&\qquad\text{ if }i\neq j\end{cases}\\
    \sum_{k=1}^n (-1)^{k+i+n+1}{k-1 \choose n-i}{n-j\choose k-1}&=\begin{cases}1 & \qquad\text{ if }i=j\label{binomials2}\\
    0&\qquad\text{ if }i\neq j\end{cases}\\
    \sum_{k=1}^n (-1)^{k+j+n+1}{n-k \choose i-1}{n-j\choose k-1}&=(-1)^{j+n}{j-1\choose n-i}\label{binomials3}\\
    \sum_{k=1}^n (-1)^{k+i+n+1}{j-1\choose n-k}{k-1 \choose n-i}&=(-1)^{i+1}{n-j\choose i-1}.\label{binomials4}
\end{align}
\end{lemma}
\begin{proof}
We prove all four equations by induction.\\
(\ref{binomials1}): The case $n=1$ is clear. If $i=1$ we have 
\begin{equation*}
    \sum_{k=1}^{n+1}(-1)^{k+n+j}{n+1-k\choose i-1}{j-1 \choose n+1-k}=\sum_{k=1}^{n+1}(-1)^{k+n+j}{j-1 \choose n+1-k}=\begin{cases}1&\text{ if } j=1\\0&\text{ if } j\neq 1\end{cases},
\end{equation*}
where the second equality follows from $\sum_{k=0}^n(-1)^k{n \choose k}$ equals $0$ if $n>0$ and $1$ otherwise. If $j=1$, we have
\begin{equation*}
    \sum_{k=1}^{n+1}(-1)^{k+n+j}{n+1-k\choose i-1}{j-1 \choose n+1-k}={0\choose i-1}=\begin{cases}1&\text{ if } i=1\\0&\text{ if } i\neq 1\end{cases}.
\end{equation*}
Now suppose that $i>1$. We have
\begin{align*}
    \sum_{k=1}^{n+1}(-1)^{k+n+j}{n+1-k\choose i-1}{j-1 \choose n+1-k}&=\sum_{k=1}^{n}(-1)^{k+n+j}{n+1-k\choose i-1}{j-1 \choose n+1-k}\\
    &\quad+(-1)^{j+1}{0 \choose i-1}{j-1\choose 0}\\
    &=0+\frac{j-1}{i-1}\cdot\sum_{k=1}^n(-1)^{k+n+(j-1)+1}{n-k\choose i-2}{j-2\choose n-k}\\
    &=\begin{cases}1&\text{ if } i=j\\0&\text{ if } i\neq j\end{cases}.
\end{align*}
This shows (\ref{binomials1}).\\
(\ref{binomials2}): The case $n=1$ is clear. First suppose that $i=1$. Then, the required identity can be proven by the same argument as used for (\ref{binomials1}) with $i=1$. Now, if $j=1$ and $i>1$, we get
\begin{align*}
    &(-1)^{k+i+n}{k-1 \choose n+1-i}{n\choose k-1}\\
    &\qquad=(-1)^{k+i+n}\left[{k-1 \choose n+1-i}{n-1 \choose k-1}+{k-2 \choose n+1-i}{n-1 \choose k-2}+{k-2 \choose n-i}{n-1 \choose k-2}\right].
\end{align*}
Then, after the case distinction $i=n+1$ and $i<n+1$, we find that we must have
\begin{equation*}
    \sum_{k=1}^{n+1}(-1)^{k+i+n}{k-1 \choose n+1-i}{n\choose k-1}=0.
\end{equation*}

Finally, for $i,j>1$ we obtain
\begin{align*}
    \sum_{k=1}^{n+1} (-1)^{k+(i-1)+n+1}{k-1 \choose n+1-i}{n+1-j\choose k-1}&=0+\sum_{k=1}^{n} (-1)^{k+i+n}{k-1 \choose n-(i-1)}{n-(j-1)\choose k-1}\\
    &=\begin{cases}1&\text{ if } i=j\\0&\text{ if } i\neq j\end{cases}.
\end{align*}
This shows (\ref{binomials2}).\\
(\ref{binomials3}): If $n=1$, $i=1$ or $j=1$, one can easily check see that the equation holds. If $i,j>1$, we have
\begin{align*}
    \sum_{k=1}^{n+1} (-1)^{k+j+n}{n+1-k \choose i-1}{n+1-j\choose k-1}&=0+\sum_{k=1}^{n} (-1)^{k+j+n}{n+1-k \choose i-1}{n+1-j\choose k-1}\\
    &\negphantom{hellohello}=\sum_{k=1}^{n} (-1)^{k+j+n}\left[{n-k \choose i-1}{n+1-j\choose k-1}+{n-k \choose i-2}{n+1-j\choose k-1}\right]\\
    &\negphantom{hellohello}=(-1)^{j-1+n}\left[{j-2\choose n-i}+{j-2\choose n-i-1}\right]\\
    &\negphantom{hellohello}=(-1)^{j+n+1}{j-1\choose n-i}.
\end{align*}
This shows (\ref{binomials3}).\\
(\ref{binomials4}): First we apply the transformation $k'=n+1-k$, so that we have to show
\begin{equation*}
    \sum_{k'=1}^n(-1)^{k'+i}{j-1 \choose k'-1}{n-k'\choose i-k'}=(-1)^{i+1}{n-j\choose i-1}.
\end{equation*}
The cases where $n=1$, $i=1$ or $i=n$ are easily deducible. So we may assume that $1<i<n+1$. We have
\begin{align*}
    \sum_{k'=1}^{n+1}(-1)^{k'+i}{j-1 \choose k'-1}{n+1-k'\choose i-k'}&=0\!+\!\sum_{k=1}^n(-1)^{k+i}\left[{j-1 \choose k'-1}{n-k'\choose i-k'}\!+\!{j-1 \choose k'-1}{n-k'\choose i-1-k'}\right]\\
    &=(-1)^{i+1}{n-j\choose i-1}-(-1)^{i}{n-j\choose i-2}\\
    &=(-1)^{i+1}{n+1-j\choose i-1}.
\end{align*}
This shows (\ref{binomials4}) and finishes the proof of the lemma.
\end{proof}
\noindent\textbf{Remark. } In the previous lemma, identities (\ref{binomials3}) and (\ref{binomials4}) satisfy a nice duality. Namely, if we apply the (unusual) transformation ${a\choose b}\mapsto{b\choose a}$, identity (\ref{binomials3}) becomes (\ref{binomials4}), up to a sign.\\

This proof of Theorem \ref{oddweightl=2} followed from the natural approach of directly studying the relations generated by the stuffle and shuffle products. Theorem \ref{oddweightl=2} can also be proved in a less explicit way. This proof was found at a later stage, and only after a suggestion by H. Bachmann.\\

For integers $S\geq l$ and $R\geq 0$ we define the generating functions
\begin{align*}
    \sgen{X_1,X_2\\Y_1,Y_2}&=\sum_{\substack{s_1,s_2>0\\ s_1+s_2=S\\r_1,r_2\geq 0\\r_1+r_2=R}}\bbracket{s_1,s_2\\r_1,r_2}X_1^{s_1-1} X_2^{s_2-1}Y_1^{r_1} Y_2^{r_2},\\
    \pgen{X_1,X_2\\Y_1,Y_2}&=\sum_{\substack{s_1,s_2>0\\ s_1+s_2=S\\r_1,r_2\geq 0\\r_1+r_2=R}}\bbracket{s_1\\r_1}\bbracket{s_2\\r_2}X_1^{s_1-1}X_2^{s_2-1}Y_1^{r_1} Y_2^{r_2}.
\end{align*}
These functions satisfy the following relations (modulo length $1$ bi-brackets)
\begin{align*}
    \pgen{X_1,X_2\\Y_1,Y_2}&\equiv\sgen{X_1,X_2\\Y_1,Y_2}+\sgen{X_2,X_1\\Y_2,Y_1}\\
    \pgen{X_1,X_2\\Y_1,Y_2}&\equiv\sgen{X_1+X_2,X_1\\Y_2,Y_1-Y_2}+\sgen{X_1+X_2,X_2\\Y_1,Y_2-Y_1}.
\end{align*}
Now let us assume that $S+R$ is odd, so that we also have
\begin{equation*}
    \sgen{-X_1,-X_2\\-Y_1,-Y_2}=-\sgen{X_1,X_2\\Y_1,Y_2},\qquad\qquad\qquad\pgen{-X_1,-X_2\\-Y_1,-Y_2}=-\pgen{X_1,X_2\\Y_1,Y_2}.
\end{equation*}
After repeated use of these properties, we obtain the following relation
\begin{align}
\label{eq:srgeneratorsl=2}
    \sgen{X_1,X_2\\Y_1,Y_2}&=\frac{1}{2}\left(\pgen{X_1,X_2\\Y_1,Y_2}-\pgen{-X_1,X_2\\-Y_1,Y_2}+\pgen{X_2-X_1,X_2\\-Y_1,Y_1+Y_2}-\pgen{X_1-X_2,X_1\\-Y_2,Y_1+Y_2}\right.
    \\&\left.\qquad+\pgen{X_1-X_2,X_2\\Y_1,Y_1+Y_2}-\pgen{X_2-X_1,X_1\\Y_2,Y_1+Y_2}\right).
\end{align}
Then Theorem \ref{oddweightl=2} follows from Theorem \ref{bdisinmdl=1}, and by considering the coefficients at $X_1^{s_1-1}X_2^{s_2-1}Y_1^{r_1}\cdots Y_2^{r_2}$.

\section{Discussion}
Theorem \ref{oddweightl=2} creates some questions. For example, what happens to the relations we obtain between bi-brackets, if we let $q\to 1^-$. We could also wonder what happens for even total weight. Finally, what happens when we consider a length $l>2$?

\subsection{The parity relation for multiple zeta values}
\label{subsection:paritymzv}
Let us consider the case $S=5$ and $R=0$. We have the equations
\begin{alignat*}{5}
    \bbracket{4}\bbracket{1}&=\hphantom{2}\bbracket{4,1}+&&&&&&\bbracket{1,4}&&\qquad+\bbracket{5}-\frac{1}{2}\bbracket{4}+\frac{1}{12}\bbracket{3},\\
    \bbracket{3}\bbracket{2}&=&&\hphantom{3}\bbracket{3,2}+&&\bbracket{2,3}&&&&\qquad+\bbracket{5}-\frac{1}{12}\bbracket{3},\\
    \bbracket{4}\bbracket{1}&= 2\bbracket{4,1}+&&\hphantom{3}\bbracket{3,2}+&&\bbracket{3,2}+&&\bbracket{1,4}&&\qquad+\bbracket{4\\1}-\bbracket{4},\\
    \bbracket{3}\bbracket{2}&= 6\bbracket{4,1}+&&3\bbracket{3,2}+&&\bbracket{3,2}&&&&\qquad+3\bbracket{4\\1}-3\bbracket{4}.
\end{alignat*}
This system can be solved (directly, or with the use of the proof of Theorem \ref{oddweightl=2}) for the brackets $\bbracket{4,1},\dots,\bbracket{1,4}$. We obtain the identities
\begin{align*}
    \bbracket{4,1}&=-\bbracket{3}\bbracket{2}+2\bbracket{5}+\frac{1}{2}\bbracket{4}-\bbracket{4\\1},\\
    \bbracket{3,2}&=3\bbracket{3}\bbracket{2}-\frac{11}{2}\bbracket{5}-\frac{1}{24}\bbracket{3}+\frac{3}{2}\bbracket{4\\1},\\
    \bbracket{2,3}&=-2\bbracket{3}\bbracket{2}+\frac{9}{2}\bbracket{5}+\frac{1}{8}\bbracket{3}-\frac{3}{2}\bbracket{4\\1},\\
    \bbracket{1,4}&=\bbracket{1}\bbracket{4}+\bbracket{2}\bbracket{3}-3\bbracket{5}-\frac{1}{12}\bbracket{3}+\bbracket{4\\1}.
\end{align*}
Since we have
\begin{equation*}
    \lim_{q\to1^-}(1-q)^N\bbracket{s_1,\dots,s_l\\r_1,\dots,r_l}=0
\end{equation*}
whenever $s_1+\dots+s_l<N$, all brackets of upper weight less than $N=5$ vanish when we consider the limit $\lim_{q\to 1^-}(1-q)^5\;\cdots\;$ of the first three equations. We obtain the following identities between multiple zeta values:
\begin{align*}
    \zeta(4,1)=-\zeta(3)\zeta(2)+2\zeta(5),  \qquad  \zeta(3,2)= 3\zeta(3)\zeta(2)-\frac{11}{2}\zeta(5),\qquad    \zeta(2,3)= -2\zeta(3)\zeta(2)+\frac{9}{2}\zeta(5).
\end{align*}

These identities between multiple zeta values are known as \textit{parity} relations. They are characterised by a single MZV of length $l$ on the left-hand side, and a rational linear combination of MZVs of length smaller than $l$ on the right-hand side. The existence of these relations depends on the parity of the length and the total weight of a multiple zeta value. It is a classical result by Tsumura \cite[Cor. 8, p. 333]{parity} and Ihara, Kaneko \& Zagier \cite{parity2}.

\begin{theorem}[Parity for multiple zeta values]
\label{paritymzv}
Let $\s=(s_1,\dots,s_l)$ be an admissible index. If $s_1+\dots+s_l+l$ is odd, then $\zeta(\s)$ can be written as a $\Q$-linear combination of multiple zeta values of length smaller than $l$. 
\end{theorem}
Similarly to what we did in the example above, we see that the case $l=2$ of Theorem \ref{paritymzv} follows from the proof of Theorem \ref{oddweightl=2} by considering the system
\begin{align*}
    \bbracket{z_{s_1,r_1}}\bbracket{z_{s_2,r_2}}&-{r_1+r_2\choose r_1}\bbracket{z_{s_1+s_2,r_1+r_2}}\equiv \bbracket{z_{s_1,r_1}z_{s_2,r_2}}+\bbracket{z_{s_2,r_2}z_{s_1,r_1}},\\
    \bbracket{z_{s_1,r_1}}\bbracket{z_{s_2,r_2}}&\equiv\sum_{\substack{1\leq l\leq s_1\\ 0\leq k\leq r_2}}{s_1+s_2-1-l \choose s_2-1}{r_1+r_2-k \choose r_1}(-1)^{r_2-k}\bbracket{z_{s_1+s_2-l,k}z_{l,r_1+r_2-k}}\\
    &\quad+\sum_{\substack{1\leq l\leq s_2\\ 0\leq k\leq r_1}}{s_1+s_2-1-l \choose s_1-1}{r_1+r_2-k \choose r_2}(-1)^{r_1-k}\bbracket{z_{s_1+s_2-l,k}z_{l,r_1+r_2-k}},
\end{align*}
which is the usual system generated by the stuffle and shuffle products modulo brackets of upper weight $<s_1+s_2$.\\

This means that we can interpret the solubility of the system generated by the stuffle and shuffle product, as the parity relation of length $2$ for bi-brackets. More on this in Section \ref{subsection:lengthl>2}.

\subsection{Even total weight}
Let us consider the system of equations
\begin{align}
\label{eq:evensystem}
    0\equiv \bbracket{z_{s_1,r_1}\astt z_{s_2,r_2}},\qquad\qquad   0\equiv \bbracket{z_{s_1,r_1}\shuf z_{s_2,r_2}},
\end{align}
for letters $z_{s_1,r_1}$ and $z_{s_2,r_2}$ with fixed upper weight $s_1+s_2=S$ and lower weight $r_1+r_2=R$. We saw in the example where $S=3$ and $R=1$ that this system cannot be solved for the individual brackets $\bbracket{z_{s_1,r_1}z_{s_2,r_2}}$. A computer calculation suggests that this is (almost) always the case for even total weight $S+R$. This means that the stuffle and shuffle product between length $1$ brackets do generate enough independent relations. Let $U(S,R)$ be the matrix corresponding to System (\ref{eq:evensystem}). This system consists of $(S-1)(R+1)$ equations, so that the quantity 
\begin{equation*}
    \lambda(S,R)\coloneqq(S-1)(R+1)-\rank(U(S,R))
\end{equation*}
encodes the deficiency of relations of System (\ref{eq:evensystem}) for it to be solvable. These values are given in Table \ref{table:dimkermetnullenU}.
\begin{table}[H]
\begin{tabular}{l|rrrrrrrrrrrrrrrrrrrrr}
S\textbackslash{}R & 0 & 1 & 2 & 3 & 4 & 5  & 6  & 7  & 8  & 9  & 10 & 11 & 12 & 13 & 14 & 15 & 16  & 17 & 18 & 19      \\ \hline
2                  & 0 & 0 & 0 & 0 & 0  & 0  & 1  & 0  & 1  & 0  & 1  & 0  & 2  & 0  & 2  & 0  & 2  & 0  & 3  & 0       \\
3                  & 0 & 1 & 0 & 1 & 0  & 2  & 0  & 3  & 0  & 3  & 0  & 4  & 0  & 5  & 0  & 5  & 0  & 6  & 0  & 7       \\
4                  & 0 & 0 & 1 & 0 & 2  & 0  & 3  & 0  & 4  & 0  & 5  & 0  & 6  & 0  & 7  & 0  & 8  & 0  & 9  & 0       \\
5                  & 0 & 1 & 0 & 3 & 0  & 4  & 0  & 5  & 0  & 7  & 0  & 8  & 0  & 9  & 0  & 11 & 0  & 12 & 0  & 13      \\
6                  & 0 & 0 & 2 & 0 & 4  & 0  & 5  & 0  & 7  & 0  & 9  & 0  & 10 & 0  & 12 & 0  & 14 & 0  & 15 & 0       \\
7                  & 0 & 2 & 0 & 4 & 0  & 6  & 0  & 8  & 0  & 10 & 0  & 12 & 0  & 14 & 0  & 16 & 0  & 18 & 0  & 20      \\
8                  & 1 & 0 & 3 & 0 & 5  & 0  & 8  & 0  & 10 & 0  & 12 & 0  & 15 & 0  & 17 & 0  & 19 & 0  & 22 & 0       \\
9                  & 0 & 3 & 0 & 5 & 0  & 8  & 0  & 11 & 0  & 13 & 0  & 16 & 0  & 19 & 0  & 21 & 0  & 24 & 0  & 27      \\
10                 & 1 & 0 & 4 & 0 & 7  & 0  & 10 & 0  & 13 & 0  & 16 & 0  & 19 & 0  & 22 & 0  & 25 & 0  & 28 & 0       \\
11                 & 0 & 3 & 0 & 7 & 0  & 10 & 0  & 13 & 0  & 17 & 0  & 20 & 0  & 23 & 0  & 27 & 0  & 30 & 0  & 33      \\
12                 & 1 & 0 & 5 & 0 & 9  & 0  & 12 & 0  & 16 & 0  & 20 & 0  & 23 & 0  & 27 & 0  & 31 & 0  & 34 & 0       \\
13                 & 0 & 4 & 0 & 8 & 0  & 12 & 0  & 15 & 0  & 20 & 0  & 24 & 0  & 28 & 0  & 32 & 0  & 36 & 0  & 40      \\
14                 & 2 & 0 & 6 & 0 & 10 & 0  & 15 & 0  & 19 & 0  & 23 & 0  & 28 & 0  & 32 & 0  & 36 & 0  & 41 & 0       
\end{tabular}
\caption{$\lambda(S,R)$.}
\label{table:dimkermetnullenU}
\end{table}
Note that $\lambda(S,R)=0$ for odd $S+R$. Table \ref{table:dimkerzondernullenU} shows the values for $\lambda(S,R)$ without the cases of odd total weight.
\setlength\tabcolsep{3.5pt}
\begin{table}[H]
\captionsetup{justification=centering,margin=1cm}
\begin{tabular}{l|rrrrrrrrrrrrrrrrrrrrrrrrr}
S=2  & 0& 0& 0& 1 & 1 & 1 & 2 & 2 & 2 & 3 & 3& 3& 4& 4& 4  & 5  & 5  & 5  & 6  & 6  & 6  & 7  & 7  & 7  & 8\\
3  & 1& 1& 2& 3 & 3 & 4 & 5 & 5 & 6 & 7 & 7& 8& 9& 9& 10 & 11 & 11 & 12 & 13 & 13 & 14 & 15 & 15 & 16 & 17 \\
4  & 0& 1& 2& 3 & 4 & 5 & 6 & 7 & 8 & 9 & 10 & 11 & 12 & 13 & 14 & 15 & 16 & 17 & 18 & 19 & 20 & 21 & 22 & 23 & 24 \\
5  & 1& 3& 4& 5 & 7 & 8 & 9 & 11 & 12 & 13 & 15 & 16 & 17 & 19 & 20 & 21 & 23 & 24 & 25 & 27 & 28 & 29 & 31 & 32 & 33 \\
6  & 0& 2& 4& 5 & 7 & 9 & 10 & 12 & 14 & 15 & 17 & 19 & 20 & 22 & 24 & 25 & 27 & 29 & 30 & 32 & 34 & 35 & 37 & 39 & 40 \\
7  & 2& 4& 6& 8 & 10 & 12 & 14 & 16 & 18 & 20 & 22 & 24 & 26 & 28 & 30 & 32 & 34 & 36 & 38 & 40 & 42 & 44 & 46 & 48 & 50 \\
8  & 1& 3& 5& 8 & 10 & 12 & 15 & 17 & 19 & 22 & 24 & 26 & 29 & 31 & 33 & 36 & 38 & 40 & 43 & 45 & 47 & 50 & 52 & 54 & 57 \\
9  & 3& 5& 8& 11 & 13 & 16 & 19 & 21 & 24 & 27 & 29 & 32 & 35 & 37 & 40 & 43 & 45 & 48 & 51 & 53 & 56 & 59 & 61 & 64 & 67 \\
10 & 1& 4& 7& 10 & 13 & 16 & 19 & 22 & 25 & 28 & 31 & 34 & 37 & 40 & 43 & 46 & 49 & 52 & 55 & 58 & 61 & 64 & 67 & 70 & 73 \\
11 & 3& 7& 10& 13 & 17 & 20 & 23 & 27 & 30 & 33 & 37 & 40 & 43 & 47 & 50 & 53 & 57 & 60 & 63 & 67 & 70 & 73 & 77 & 80 & 83 \\
12 & 1& 5& 9& 12 & 16 & 20 & 23 & 27 & 31 & 34 & 38 & 42 & 45 & 49 & 53 & 56 & 60 & 64 & 67 & 71 & 75 & 78 & 82 & 86 & 89 \\
13 & 4& 8& 12& 16 & 20 & 24 & 28 & 32 & 36 & 40 & 44 & 48 & 52 & 56 & 60 & 64 & 68 & 72 & 76 & 80 & 84 & 88 & 92 & 96 & 100 \\
14 & 2& 6& 10& 15 & 19 & 23 & 28 & 32 & 36 & 41 & 45 & 49 & 54 & 58 & 62 & 67 & 71 & 75 & 80 & 84 & 88 & 93 & 97 & 101 & 106 
\end{tabular}
\caption{$\lambda(S,R)$ for $S+R$ odd. The horizontal values are increases of $R$ by $2$, starting at $R=0$ or $R=1$, if $S$ is even or odd, respectively.}
\label{table:dimkerzondernullenU}
\end{table}
\subsection{Length $l>2$ and the parity relation for bi-brackets}
\label{subsection:lengthl>2}
We could try to reproduce the methods used in the case $l=2$, but applied to words of length $l>2$. For example, let us look at the smallest case: $l=3$. Fix $S\geq 3$ and $R\geq 0$. We consider all products 
\begin{align*}
    \bbracket{z_{s_1,r_1}z_{s_2,r_2}}\bbracket{z_{s_3,r_3}}&=\bbracket{z_{s_1,r_1}z_{s_2,r_2}\ast z_{s_3,r_3}},\\
    \bbracket{z_{s_1,r_1}z_{s_2,r_2}}\bbracket{z_{s_3,r_3}}&=\bbracket{z_{s_1,r_1}z_{s_2,r_2}\shuffle z_{s_3,r_3}},
\end{align*}
 with $s_1+s_2+s_3=S$ and $r_1+r_2+r_3=R$. However, we do not know whether $\bbracket{z_{s_1,r_1}z_{s_2,r_2}}\equiv 0$ mod $\MD$ for arbitrary $s_1,s_2,r_1,r_2$ (we only know this for $S+R$ odd). Let us first assume that this is indeed the case, so that we can ignore the involved terms of length $2$. We define the products $\astt$ and $\shuf$ to be $\ast$ and $\shuffle$, respectively, but without all terms of lower length, i.e., those of length $1$ and $2$.
\begin{lemma}
The partition relation for $l=3$ is given by
\begin{align*}
    &Z_{r_1}^{s_1}Z_{r_2}^{s_2}Z_{r_3}^{s_3}=\\
    &\sum_{j=0}^{s_3-1}\sum_{k=0}^{s_2+j-1}\sum_{n=0}^{r_1}\sum_{l=0}^{r_2+n}(-1)^{j+k}{s_2+j-k-1\choose j}{s_1+k-1\choose k}{r_3+l\choose l}{r_2+n\choose n}Z^{r_3+l+1}_{s_3-j-1}Z^{r_2+n-l+1}_{s_2+j-k-1}Z^{r_1-n+1}_{s_1+k-1},
\end{align*}
with the less spacious notation $Z_r^s= z_{s,r}$.
\end{lemma}
\proof The relation is proven similar to the relations in Corollary \ref{length1partition}.\qed \\
We also have
\begin{equation*}
    z_{s_1,r_1}z_{s_2,r_2}\diam z_{s_3,r_3}=z_{s_1,r_1}z_{s_2,r_2}z_{s_3,r_3}+z_{s_1,r_1}z_{s_3,r_3}z_{s_2,r_2}+z_{s_3,r_3}z_{s_1,r_1}z_{s_2,r_2}.
\end{equation*}
Clearly, calculating $z_{s_1,r_1}z_{s_2,r_2}\shuf z_{s_3,r_3}=P\big(P(z_{s_1,r_1}z_{s_2,r_2})\astt P(z_{s_3,r_3})\big)$ is impractical and should be avoided. However, we can look at the system given by
\begin{align}
\label{systeml=3}
    0&\equiv \bbracket{P(z_{s_1,r_1}z_{s_2,r_2})\astt P(z_{s_3,r_3})}\nonumber\\
    0&\equiv\bbracket{P(z_{s_1,r_1}z_{s_2,r_2}\astt z_{s_3,r_3})}.
\end{align}
This system gives the same relations as the usual system:
\begin{align*}
    0&\equiv \bbracket{P\big(P(z_{s_1,r_1}z_{s_2,r_2})\astt P(z_{s_3,r_3})\big)}\\
    0&\equiv \bbracket{z_{s_1,r_1}z_{s_2,r_2}\astt z_{s_3,r_3}}.
\end{align*}
A computer check suggests (for $S,R\leq 10$) that system (\ref{systeml=3}) has a solution whenever $S+R$ is even. This in contrary to the case $l=2$, where the system has a solution if $S+R$ is odd. This leaves us wondering if this pattern repeats as $l$ increases. Theorem \ref{oddweightl=2} and the limited data for $l=3$ let us suspect the following.
\begin{suspicion}
\label{suspicion:stsh1}
Fix some length $l\geq2$ and integers $S\geq l$ and $R\geq0$. Assume that $S+R+l$ is odd. Also assume that $\bbracket{z_{s_1,r_1}\cdots z_{s_t,r_t}}\equiv 0$ mod $\MD$ for all $2\leq t\leq l-1$. Then the system 
\begin{align*}
    0&\equiv \bbracket{z_{s_1,r_1}\cdots z_{s_{t-1},r_{t-1}}\astt z_{s_t,r_t}\cdots z_{s_l,r_l}}\\
    0&\equiv \bbracket{z_{s_1,r_1}\cdots z_{s_{t-1},r_{t-1}}\shuf z_{s_t,r_t}\cdots z_{s_l,r_l}},
\end{align*}
given by all words $z_{s_1,r_1}\cdots z_{s_{t-1},r_{t-1}}$ and $z_{s_t,r_t}\cdots z_{s_l,r_l}$, which satisfy $s_1+\dots+s_l=S$ and $r_1+\dots+r_l=R$, has a solution.
\end{suspicion}
Or perhaps more specifically:
\begin{suspicion}
\label{suspicion:stsh2}
Fix some length $l\geq2$ and integers $S\geq l$ and $R\geq0$. Assume that $S+R+l$ is odd. Also assume that $\bbracket{z_{s_1,r_1}\cdots z_{s_{l-1},r_{l-1}}}\equiv 0$ mod $\MD$. Then the system 
\begin{align*}
    0&\equiv \bbracket{z_{s_1,r_1}\cdots z_{s_{l-1},r_{l-1}}\astt z_{s_l,r_l}}\\
    0&\equiv \bbracket{z_{s_1,r_1}\cdots z_{s_{l-1},r_{l-1}}\shuf z_{s_l,r_l}},
\end{align*}
given by all words $z_{s_1,r_1}\cdots z_{s_{l-1},r_{l-1}}$ and $z_{s_l,r_l}$, which satisfy $s_1+\dots+s_l=S$ and $r_1+\dots+r_l=R$, has a solution.
\end{suspicion}
Now let us drop the assumption that $\bbracket{z_{s_1,r_1}z_{s_2,r_2}}\equiv 0$ for \textit{all} words $z_{s_1,r_1}z_{s_2,r_2}$. This means that in general, we do not have the relation
\begin{equation*}
    \bbracket{z_{s_1,r_1}z_{s_2,r_2}\astt z_{s_3,r_3}}=\bbracket{z_{s_1,r_1}z_{s_2,r_2}}\cdot\bbracket{z_{s_3,r_3}}\equiv 0\qquad\qquad\text{ mod }\MD.
\end{equation*}
We know for words with odd total weight (Theorem \ref{oddweightl=2}) this relation holds. This means we could consider the system of equations where we only consider words $z_{s_1,r_1}z_{s_2,r_2}$ with $s_1+s_2+r_1+r_2$ odd. However, a computer check ($S,R\leq10$) suggests that this restricted system in general has no solution (with the exception of small $S,R$).\\

Another way to solve the problem is to look at the products 
\begin{equation*}
    \bbracket{z_{s_1,r_1}\astt z_{s_2,r_2}\astt z_{s_3,r_3}}=\bbracket{z_{s_1,r_1}}\cdot\bbracket{z_{s_2,r_2}}\cdot\bbracket{z_{s_3,r_3}}\equiv 0\qquad\qquad \text{ mod }\MD.
\end{equation*}
Again, a computer check ($S,R\leq10$) tells us that this system in general has no solutions.\\

We noted in Section \ref{subsection:paritymzv} that the bi-brackets of length $2$ satisfy the parity relation. This was a result of the solubility of the system generated by stuffle and shuffle products. Suspicion \ref{suspicion:stsh1} and \ref{suspicion:stsh2} both imply this parity relation for arbitrary length.

\begin{suspicion}[Parity relation for bi-brackets]
Let $\bbracket{s_1,\dots,s_l\\r_1,\dots,r_l}$ be a bi-bracket of upper weight $S$ and lower weight $R$. Let us assume that $S+R+l$ is odd. Then $\bbracket{s_1,\dots,s_l\\r_1,\dots,r_l}$ can be expressed as a $\Q$-linear combination of bi-brackets with length less than $l$.
\end{suspicion}

\section*{Computational methods}
\addcontentsline{toc}{section}{Computational methods}
The results in this thesis were based on calculations made by the computational system \textit{Magma} \cite{magma}. This system had neither implementations for calculating multiple zeta values and their related stuffle and shuffle product, nor for calculating mono- and bi-brackets, and so they had to be made manually. If one is interested in these applications, or in the methods used for these calculations, feel free to contact me.

\phantomsection


\end{document}